\documentclass[a4paper, 11pt]{article}

\usepackage{amsmath, amssymb, amsthm} 
\usepackage{mathtools}
\usepackage[top=40truemm,bottom=30truemm,left=30truemm,right=30truemm]{geometry}
\usepackage[all]{xy}
\usepackage{enumitem}

\usepackage[bookmarksnumbered]{hyperref}

\usepackage{xcolor}

\usepackage{tikz}
\usetikzlibrary{intersections, calc, arrows.meta}
\usepackage{subcaption}

\DeclareMathOperator{\arccot}{arccot}
\DeclareMathOperator{\Arg}{Arg}
\DeclareMathOperator{\sgn}{sgn}

\newcommand{\amd}{\mspace{10mu} \text{and} \mspace{10mu}}
\newcommand{\argdot}{\mspace{1mu} \cdot \mspace{1mu}}
\newcommand{\e}{\mathrm{e}}

\newcommand{\iu}{\mathrm{i}}

\newcommand{\C}{\mathbb{C}}
\newcommand{\R}{\mathbb{R}}
\newcommand{\Z}{\mathbb{Z}}

\newcommand{\abs}[1]{\left\lvert #1 \right\rvert}
\newcommand{\der}[1]{\frac{\mathrm{d}}{\mathrm{d}#1}}

\newcommand{\Set}[2]{\left\{ \mspace{1mu} #1 : #2 \mspace{1mu} \right\}}

\theoremstyle{plain}
\newtheorem{theorem}{Theorem}[section]
\newtheorem{lemma}[theorem]{Lemma}

\newtheorem{corollary}[theorem]{Corollary}

\theoremstyle{definition}
\newtheorem{definition}[theorem]{Definition}

\newtheorem{notation}{Notation}

\theoremstyle{remark}
\newtheorem{remark}{Remark}

\numberwithin{equation}{section}

\begin{document}

\title{Stability region and critical delay}
\author{Junya Nishiguchi\thanks{Mathematical Science Group, Advanced Institute for Materials Research (AIMR), Tohoku University, Katahira 2-1-1, Aoba-ku, Sendai, 980-8577, Japan}
\footnote{E-mail: \url{junya.nishiguchi.b1@tohoku.ac.jp}}}
\date{}

\maketitle

\begin{abstract}
The location of roots of the characteristic equation of a linear delay differential equation (DDE) determines the stability of the linear DDE.
However, by its transcendency, there is no general criterion on the contained parameters for the stability.
Here we mainly concentrate on the study of a simple transcendental equation $(*)$ $z + a - w \mathrm{e}^{-\tau z} = 0$ with coefficients of real $a$ and complex $w$ and a delay parameter $\tau > 0$ to tackle this transcendency brought by delay.
The consideration is twofold: (i) to give the stability region in the parameter space for Eq.~$(*)$ by using the critical delay and (ii) to compare this with a graphical method (so-called the method of D-partitions) by combining with the delay sequence obtained by conditions for purely imaginary roots.
By (i), we obtain another proof of Hayes' and Sakata's results, which reveals the nature of imaginary $w$ case in Eq.~$(*)$.
By (ii), we propose a method combining the analytic one and geometric one.
This combination is important because it will be helpful in studying characteristic equations having higher-dimensional parameters.

\begin{flushleft}

\end{flushleft}

\begin{flushleft}
\end{flushleft}

\end{abstract}

\tableofcontents

\section{Introduction}

Retarded functional differential equations (RFDEs) give a mathematical formulation of a class of delay differential equations (DDEs).
It is well-known that for the exponential stability of the trivial solution of a given linear RFDE, it is necessary and sufficient that all the roots of the corresponding characteristic equation have negative real parts (refs.\ Hale and Verduyn Lunel~\cite{Hale--Lunel 1993} and Diekmann et al.~\cite{Diekmann--vanGils--Lunel--Walther 1995}).
Such a characteristic equation is transcendental in general, and it has infinitely many roots in principle in the complex number plane $\C$.
For this reason, it is difficult to obtain the condition on the contained parameters for which all the roots are located in the left half of the complex plane $\C$.

Many authors have elaborated on the study of the transcendental equations obtained as characteristic equations of linear RFDEs, where various studies exist depending on the nature of the time-delay structure and on the form of differential equations (e.g., higher-order equations, systems of equations, or neutral equations).
We refer the reader to St\'{e}p\'{a}n~\cite{Stepan 1989} as a general reference of the stability problem of linear RFDEs.
See e.g., \cite{Bortz 2016}, \cite{Diekmann--Getto--Nakata 2016}, \cite{Nishiguchi 2016}, \cite{Petit--Asllani--Fanelli--Lauwens--Carletti 2016}, \cite{Breda--Menegonand--Nonino 2018}, \cite{Surya--Vyasarayani--Kalmar-Nagy 2018}, \cite{Fukuda--Kiri--Saito--Ueda 2022}, and \cite{Yanchuk--Wolfrum--Pereira--Turaev preprint} for recent studies.
See also \cite{Hadeler--Ruan 2007} for linear stability analysis of partial differential equations with time-delay.
However, the understanding of a simple transcendental equation of the form
\begin{equation}\label{eq: TE}
	z + a - w\e^{-\tau z} = 0 \tag{$*$}
\end{equation}
with complex coefficients $a$ and $w$ and with a delay parameter $\tau > 0$ has not yet completed.
Here Eq.~\eqref{eq: TE} is obtained as the characteristic equation of a scalar linear DDE
\begin{equation}\label{eq: x'(t) = -ax(t) + wx(t - tau)}
	\dot{x}(t) = - ax(t) + wx(t - \tau)
		\mspace{20mu}
	(\text{$t \in \R$, $x(t) \in \C$})
\end{equation}
by assuming a complex exponential solution $x(t) = \mathrm{e}^{zt}$.
In this paper, we will call the region in the parameter space for which all the roots of Eq.~\eqref{eq: TE} have negative real parts the \textit{stability region}.
Usually, delay parameters have special natures different from those which usual control parameters have.
For this reason, we separate delay parameters from the parameter space in which stability region is considered.

The purpose of this paper is to provide a unified perspective to obtain the stability region of Eq.~\eqref{eq: TE} with real $a$ and complex $w$, where the existence and the expression of the \textit{critical delay} $\tau_\mathrm{c}(a, w)$ are essential.
Here the critical delay $\tau_\mathrm{c}(a, w)$ means the threshold $\tau$-value which divides the positive real number line into two parts so that all the roots of Eq.~\eqref{eq: TE} have negative real parts for $\tau \in (0, \tau_\mathrm{c}(a, w))$, and Eq.~\eqref{eq: TE} has a root with positive real part for $\tau \in (\tau_\mathrm{c}(a, w), \infty)$.
In \cite{Matsunaga 2008}, it has been essentially shown that the critical delay $\tau_\mathrm{c}(a, w)$ is given by
\begin{equation}\label{eq: tau_c(a, w), |Arg(w)|, arccos}
	\tau_\mathrm{c}(a, w)
	= \frac{1}{\sqrt{\abs{w}^2 - a^2}} \left[ \abs{\Arg(w)} - \arccos{\left( \frac{a}{\abs{w}} \right)} \right]
\end{equation}
for real $a$ and complex $w$ satisfying
\begin{equation*}
	w \ne 0 \amd \Re(w) < a < \abs{w}.
\end{equation*}
Here $\Re(w)$ denotes the real part of $w$, $\Arg(w) \in (-\pi, \pi]$ is the principal value of the argument of nonzero $w$, and
\begin{equation*}
	\arccos \colon [-1, 1] \to [0, \pi]
\end{equation*}
is the inverse function  of $\cos|_{[0, \pi]} \colon [0, \pi] \to [-1, 1]$.
In the following three paragraphs, we briefly review some studies related to Eq.~\eqref{eq: TE}.

The case of imaginary $a$ is a source of many interesting dynamics (e.g., see \cite{Yeung--Strogatz 1999}, \cite{Hovel--Scholl 2005}, and \cite{Fiedler--Flunkert--Georgi--Hovel--Scholl 2007}).
We note that a necessary and sufficient condition on $a$, $w$, and $\tau$ for which all the roots of Eq.~\eqref{eq: TE} have negative real parts is obtained in \cite{Nishiguchi 2016}, and it has been applied to the stabilization of an unstable  equilibrium of autonomous ordinary differential equations by the delayed feedback control proposed by Pyragas~\cite{Pyragas 1992} (cf.\ \cite{Hovel--Scholl 2005}).
Here the choice of imaginary $a$ is essential because there is no stabilization in Eq.~\eqref{eq: TE} as increasing the delay parameter $\tau$ from $0$ if $a$ is real.
However, the detailed stability region in $(a, w)$-space, which is a real $4$-dimensional space because $a$ and $w$ are complex, has not been obtained (cf.\ \cite{Breda 2012}).

The case of real $a$ and $w$ is studied by Hayes~\cite{Hayes 1950}.
Unlike the complex coefficients case, the complete picture of the stability region in $(a, w)$-plane has been obtained (refs.\ \cite[Figure 5.1 in Chapter 5]{Hale--Lunel 1993} and \cite[Figure XI.1 in Chapter XI]{Diekmann--vanGils--Lunel--Walther 1995}).
In this case, it is important that $(a, w)$-plane is $2$-dimensional.
The case of real $a$ and complex $w$ is studied by Sakata~\cite{Sakata 1998}.
Although the picture of the stability region is depicted, the statement and the proof are complicated, which makes it difficult to understand the nature of this situation.
We note that $(a, w)$-space for real $a$ and complex $w$ is $3$-dimensional, but the consideration can be reduced to regions in $(a, \abs{w})$-plane by fixing the argument of $w$.

It should be pointed out that there are incorrect results in the literature.
Borrowing one of the results and the arguments discussed by Braddock and van den Driessche~\cite{Braddock--Driessche 1975/76}, B\'{e}lair~\cite[Theorem 2.6]{Belair 1993} has discussed the stability region of Eq.~\eqref{eq: TE} in the complex $w$-plane by letting $a \coloneqq 1$.
As is already pointed out by Takada, Hori, and Hara~\cite{Takada--Hori--Hara 2014}, \cite[Theorem 2.6]{Belair 1993} is incorrect.
We note that the corresponding result in \cite{Braddock--Driessche 1975/76} is also incorrect.

In this paper, we will show that the critical delay function
\begin{equation}\label{eq: critical delay function}
	(a, w) \mapsto \tau_\mathrm{c}(a, w) \in (0, \infty)
\end{equation}
has enough power to deduce the stability region of Eq.~\eqref{eq: TE} in $(a, w)$-space by solving the inequality
\begin{equation}\label{eq: inequality on critical delay}
	\tau_\mathrm{c}(a, w) > \tau
\end{equation}
with respect to $(a, w)$ for the case that $a$ is real and $w$ is complex.
We call the corresponding method the \textit{method by critical delay}.
The idea of solving inequality~\eqref{eq: inequality on critical delay} is to use the following expression of the critical delay:
\begin{equation}\label{eq: tau_c(a, w), |Arg(w)|, arccot}
	\tau_\mathrm{c}(a, w)
	= \frac{1}{\sqrt{\abs{w}^2 - a^2}} \left[ \abs{\Arg(w)} - \arccot{\left( \frac{a}{\sqrt{\abs{w}^2 - a^2}} \right)} \right].
\end{equation}
Here $\cot{\theta} \coloneqq \cos{\theta}/{\sin{\theta}}$ for $\theta \not\in \pi\Z$, and $\arccot \colon \R \to (0, \pi)$ denotes the inverse function of $\cot|_{(0, \pi)} \colon (0, \pi) \to \R$.

To see the effectiveness of the method by critical delay, we first concentrate our consideration on the real $a$, $w$ case and give another proof of Hayes' result by using the critical delay function.
From expression~\eqref{eq: tau_c(a, w), |Arg(w)|, arccot}, it will be turned out that solving an inequality
\begin{equation*}
	C(\theta)
	\coloneqq \theta \cot\theta
	< r
\end{equation*}
with respect to $\theta \in (0, \pi)$ is essential, and by the monotonicity of the function $C \colon (0, \pi) \to (-\infty, 1)$, the condition is simply expressed by the inverse function $C^{-1} \colon (-\infty, 1) \to (0, \pi)$.
This may be elementary but will give an insight into the analysis of real $a$ and imaginary $w$ case.
We next move to the consideration of the real $a$ and imaginary $w$ case and give another proof of Sakata's result by using the critical delay function.
Expression~\eqref{eq: tau_c(a, w), |Arg(w)|, arccot} also leads us to solve an inequality
\begin{equation}\label{eq: theta cot(theta - varphi) < r}
	C(\theta; \varphi)
	\coloneqq
	\theta \cot(\theta - \varphi) < r
\end{equation}
for $\theta \in (0, \varphi)$, and it plays an essential role to obtain the stability region in $(a, \abs{w})$-plane.
Here $\varphi \in (0, \pi)$ corresponds to the absolute value of the principal value of the argument $\Arg(w)$ of $w \in \C \setminus \R$, and the real parameter $r$ will be given appropriately.
We note that the real $w$ case can be considered as a limiting case of $\varphi \uparrow \pi$.

The above mentioned approach for the proof of Sakata's result is simple but needs elaborative calculations to obtain the behavior of the function $C(\argdot; \varphi)$.
In the literature, there is another method to obtain the stability regions of transcendental equations, which is the so called \textit{method of D-partitions} (refs.\ \`{E}l'sgol'ts and Norkin~\cite{Elsgolts--Norkin 1973}, Kolmanovski\u{i} and Nosov~\cite{Kolmanovskii--Nosov 1986}).
Basically, this is a method to obtain hyper-surfaces in the parameter space by considering the condition on the parameters under which a given transcendental equation has a root $\iu \Omega$ for some real number $\Omega$.
Here $\iu$ is the imaginary unit, and the real number $\Omega$ corresponds to the angular frequency of the corresponding periodic solutions (i.e., $T = 2\pi/{\abs{\Omega}}$ for the period $T > 0$ of the periodic solutions\footnote{When $\Omega = 0$, $T$ is interpreted as $\infty$. In this case, the periodic solution is constant.}).

In this paper, we will also compare the above mentioned another proof of Sakata's result based on the method by critical delay with the method of D-partitions.
A direct calculation will show that for each fixed $\Arg(w)$, $a$ and $\abs{w}$ are parametrized by the angular frequency $\Omega$.
Here we have two distinct points from the real $w$ case: (i) The property that $\iu\Omega$ is a root of Eq.~\eqref{eq: TE} does not necessarily imply that its complex conjugate $-\iu\Omega$ is a root of Eq.~\eqref{eq: TE}.
Therefore, it is insufficient to only consider the case $\Omega > 0$.
(ii) $\abs{w}$ should be kept positive.

The main ingredient in this paper for the study of the stability region of Eq.~\eqref{eq: TE} via the method of D-partitions is to connect the curves parametrized by angular frequency with the sequence composed of the $\tau$-values for which Eq.~\eqref{eq: TE} has purely imaginary roots.
These $\tau$-values are essentially obtained in \cite{Matsunaga 2008}, but there is an ambiguity because two sequences composed of $\tau$-values are given in \cite{Matsunaga 2008}.
The above mentioned connection gives a ``one-to-one correspondence'' between the curves and the $\tau$-values of the sequence.
Furthermore, it naturally produce an ``ordering'' of the curves via the ordering of the $\tau$-values.

This paper is organized as follows.
In Section~\ref{sec: method by critical delay}, we give a detailed explanation of the method by critical delay for Eq.~\eqref{eq: TE} with $a \in \R$ and $w \in \C$.
Here the domain of definition of critical delay function~\eqref{eq: critical delay function} is also determined.
In Section~\ref{sec: dependence of roots on delay parameter}, to give a general insight into the method by critical delay, we study the dependence of real parts of roots of Eq.~\eqref{eq: TE} on the delay parameter $\tau > 0$ for the case of $a \in \C$ and $w \in \C \setminus \{0\}$.
In Subsection~\ref{subsec: crossing direction of a purely imaginary root}, we examine the crossing direction of a purely imaginary root $\iu\Omega_0$ for some $\Omega_0 \in \R \setminus \{0\}$ with respect to the delay parameter $\tau > 0$.
Then the obtained result will reveal that its crossing direction is determined by the sign of
\begin{equation*}
	\Omega_0(\Omega_0 + \Im(a)),
\end{equation*}
where $\Im(a)$ is the imaginary part of $a$.
The above value is not necessarily positive if $\Im(a) \ne 0$.
Therefore, the imaginary $a$ case will be more complicated.
For this reason, we assume that $a$ is real in the later sections.
In Section~\ref{sec: purely imaginary roots}, we will find conditions on $\tau$ for which Eq.~\eqref{eq: TE} has a purely imaginary root to know what happens for general parameter values.
Sections~\ref{sec: method by critical delay for real a and w} and \ref{sec: method by critical delay for real a and imaginary w} are devoted to the investigation of the stability region of Eq.~\eqref{eq: TE} by the method by critical delay, where another proof of Hayes' and Sakata's results are given.
In Section~\ref{sec: comparison with D-partitions}, we discuss the comparison between the method by critical delay and the method of D-partitions.
In Section~\ref{sec: discussions}, we will discuss the imaginary $a$ case and the case of multiple delays to contribute to possible future researches.
Appendix~\ref{sec: Lambert W function} gives a proof of the key theorem in this paper via the Lambert $W$ function.

\subsection*{Notations}

Let $\iu$ denote the imaginary unit.
For a complex number $z$, its complex conjugate is denoted by $\bar{z}$.
When $z \ne 0$, let $\Arg(z)$ be the \textit{principal value of the argument}, i.e., $\Arg(z) \in (-\pi, \pi]$ and $z = {\abs{z}}\e^{\iu \Arg(z)}$.
Throughout this paper, we will use the following notation.

\begin{notation}
For each $a, w \in \C$, let $T(a, w)$ denote the set of all $\tau > 0$ for which all the roots of Eq.~\eqref{eq: TE}
\begin{equation*}
	z + a - w\e^{-\tau z} = 0
\end{equation*}
have negative real parts.
\end{notation}

When $w = 0$, we have
\begin{equation*}
	T(a, 0) =
	\begin{cases}
	(0, \infty) & (\Re(a) > 0), \\
	\emptyset & (\Re(a) \le 0)
	\end{cases}
\end{equation*}
because Eq.~\eqref{eq: TE} becomes $z + a = 0$.

\section{Method by critical delay for real $a$ and complex $w$}\label{sec: method by critical delay}

In this section, we study Eq.~\eqref{eq: TE}
\begin{equation*}
	z + a - w\e^{-\tau z} = 0
\end{equation*}
for given $a \in \R$, $w \in \C$, and $\tau > 0$.
The basics of our consideration is the following theorem.

\begin{theorem}[cf.\ \cite{Matsunaga 2008}]\label{thm: T(a, w), real a}
Suppose $a \in \R$ and $w \in \C$.
Then the following statements hold:
\begin{enumerate}[label=\textup{(\Roman*)}]
\item $T(a, w) = (0, \infty)$ if and only if $a \ge \abs{w}$ and $a > \Re(w)$.
\item $T(a, w)$ is a nonempty proper subset of $(0, \infty)$ if and only if
\begin{equation*}
	w \ne 0 \amd \Re(w) < a < \abs{w}.
\end{equation*}
In this case, $T(a, w) = (0, \tau_\mathrm{c}(a, w))$ holds, where $\tau_\mathrm{c}(a, w) > 0$ is expressed by \eqref{eq: tau_c(a, w), |Arg(w)|, arccos}
\begin{equation*}
	\tau_\mathrm{c}(a, w)
	= \frac{1}{\sqrt{\abs{w}^2 - a^2}} \left[ \abs{\Arg(w)} - \arccos{\left( \frac{a}{\abs{w}} \right)} \right]
\end{equation*}
or \eqref{eq: tau_c(a, w), |Arg(w)|, arccot}
\begin{equation*}
	\tau_\mathrm{c}(a, w)
	= \frac{1}{\sqrt{\abs{w}^2 - a^2}} \left[ \abs{\Arg(w)} - \arccot{\left( \frac{a}{\sqrt{\abs{w}^2 - a^2}} \right)} \right].
\end{equation*}
\item $T(a, w)$ is empty if and only if $a \le \Re(w)$.
\end{enumerate}
\end{theorem}

\begin{remark}\label{rmk: absolute stability}
The condition $a \ge \abs{w}$ and $a > \Re(w)$ in (I) is equivalent to
\begin{equation*}
	a \in
	\begin{cases}
	(\abs{w}, \infty) & (\Arg(w) = 0), \\
	\bigl[ \abs{w}, \infty \bigr) & (\text{otherwise}).
	\end{cases}
\end{equation*}
Therefore, it is also equivalent to $a \ge \abs{w}$ and $a \ne w$.
\end{remark}

Theorem~\ref{thm: T(a, w), real a} except \eqref{eq: tau_c(a, w), |Arg(w)|, arccot} is obtained by Matsunaga~\cite[Theorem 2]{Matsunaga 2008} for the imaginary $w$ case with a different expression of the critical delay.
See the latter discussion for the comparison.
We note that expression~\eqref{eq: tau_c(a, w), |Arg(w)|, arccot} is obtained by using an identity
\begin{equation*}
	\arccos(x) = \arccot{\left( \frac{x}{\sqrt{1 - x^2}} \right)}
	\mspace{20mu} (x \in (-1, 1)),
\end{equation*}
which is obtained from
\begin{equation*}
	\cot(\arccos x) = \frac{x}{\sqrt{1 - x^2}}
	\mspace{20mu} (x \in (-1, 1)).
\end{equation*}

We will call the subsets corresponding to the cases (I), (II), and (III) in Theorem~\ref{thm: T(a, w), real a} the \textit{delay-independent stability region}, \textit{delay-dependent stability or instability region}, and \textit{delay-independent instability region}, respectively.
See Fig.~\ref{fig: decomposition of (a, |w|)-plane by critical delay} for an illustration of the each regions in $(a, \abs{w})$-plane obtained by Theorem~\ref{thm: T(a, w), real a} when $\Arg(w) = 0$ and $\Arg(w) = 3\pi/4$, respectively.
Here we are considering $\Arg(w)$ as a parameter.

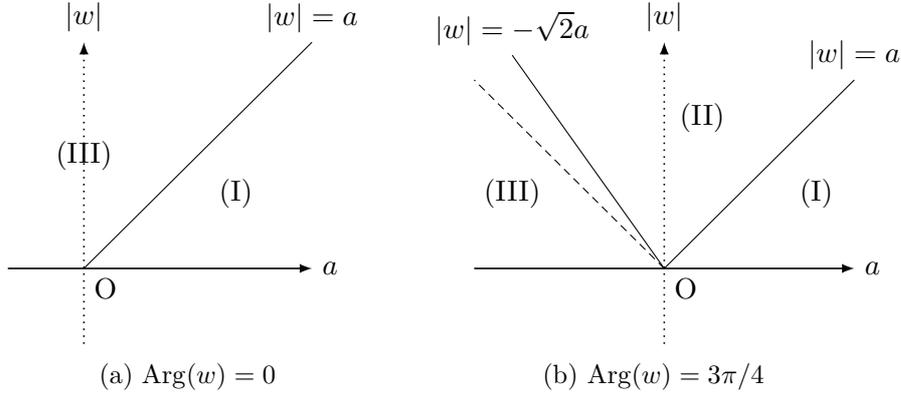
\begin{figure}[tbp]
\centering
	\begin{subfigure}{0.4\columnwidth}
	\centering
	\begin{tikzpicture}
		\draw[-latex, semithick] (-1,0) -- (3,0) node[right]{$a$}; 
		\draw[-latex, semithick, dotted] (0,-1) -- (0,3) node[above]{$\abs{w}$}; 
		\draw (0,0) node[below right]{O}; 
		\draw[domain=0:3, variable=\a] plot(\a,\a) node[above]{$\abs{w} = a$};
		\draw (2,1) node{(I)};
		\draw (0,1.5) node{(III)};
	\end{tikzpicture}
	\caption{$\Arg(w) = 0$}
	\end{subfigure}
	\begin{subfigure}{0.4\columnwidth}
	\centering
	\begin{tikzpicture}
		\draw[-latex, semithick] (-2.5,0) -- (2.5,0) node[right]{$a$}; 
		\draw[-latex, semithick, dotted] (0,-1) -- (0,3) node[above]{$\abs{w}$}; 
		\draw (0,0) node[below right]{O}; 
		\draw[domain=0:2.5, variable=\a] plot(\a,\a) node[above]{$\abs{w} = a$};
		\draw[domain=0:-2, variable=\a] plot(\a,{-sqrt(2)*\a}) node[above]{$|w| = -\sqrt{2}a$};
		\draw[domain=0:-2.5, variable=\a, densely dashed] plot(\a,-\a); 
		\draw (2,1) node{(I)};
		\draw (0.5,2) node{(II)};
		\draw (-2,1) node{(III)};
	\end{tikzpicture}
	\caption{$\Arg(w) = 3\pi/4$}
	\end{subfigure}
\caption{Decompositions of $(a, \abs{w})$-plane obtained by Theorem~\ref{thm: T(a, w), real a} for the cases $\Arg(w) = 0$ and $\Arg(w) = 3\pi/4$. The dashed line in (b) expresses the graph of $\abs{w} = -a$.}
\label{fig: decomposition of (a, |w|)-plane by critical delay}
\end{figure}

\subsection{Comparison with Matsunaga's expression of critical delay}

In \cite{Matsunaga 2008}, the critical delay is given by
\begin{equation}\label{eq: Matsunaga's critical delay}
	\frac{\sgn(b)}{\sqrt{b^2 - a^2}} \left[ \arccos{\left(-\frac{a}{b} \right)} - \abs{\theta} \right],
\end{equation}
where the parameter $w$ corresponds to $-b\e^{\iu\theta}$ with $b \in \R \setminus \{0\}$ and $-\pi/2 \le \theta \le \pi/2$, and $\sgn(b) \coloneqq b/{\abs{b}}$.
The value~\eqref{eq: Matsunaga's critical delay} is indeed equal to the right-hand side of \eqref{eq: tau_c(a, w), |Arg(w)|, arccos} as the following lemma shows.

\begin{lemma}\label{lem: critical delay in Matsunaga}
Let $w = -b\e^{\iu\theta}$ for $b \in \R \setminus \{0\}$ and $\theta \in [-\pi/2, \pi/2]$.
Then
\begin{equation*}
	\frac{\sgn(b)}{\sqrt{b^2 - a^2}} \left[ \arccos{\left(-\frac{a}{b} \right)} - \abs{\theta} \right]
	= \frac{1}{\sqrt{\abs{w}^2 - a^2}} \left[ \abs{\Arg(w)} - \arccos{\left( \frac{a}{\abs{w}} \right)} \right]
\end{equation*}
holds.
\end{lemma}

\begin{proof}
The proof is divided into the following two cases.
\begin{itemize}[leftmargin=*]
\item Case 1: $b > 0$.
Since ${\abs{w}}\e^{\iu \Arg(w)} = b\e^{\iu(\theta + \pi)}$, we have
\begin{equation*}
	\abs{w} = b \amd
	\Arg(w) =
	\begin{cases}
		\theta - \pi & (0 < \theta \le \pi/2), \\
		\theta + \pi & (-\pi/2 \le \theta \le 0),
	\end{cases}
\end{equation*}
where $\Arg(w) < 0$ for $0 < \theta \le \pi/2$ and $\Arg(w) > 0$ for $-\pi/2 \le \theta \le 0$.
Therefore,
\begin{align*}
	-{\abs{\theta}}
	&=
	\begin{cases}
		-\Arg(w) - \pi & (0 < \theta \le \pi/2), \\
		\Arg(w) - \pi & (-\pi/2 \le \theta \le 0)
	\end{cases} \\
	&= \abs{\Arg(w)} - \pi.
\end{align*}
By using this and an identity
\begin{equation}\label{eq: arccos(x) + arccos(-x) = pi}
	\arccos(x) + \arccos(-x) = \pi
	\mspace{20mu}
	(x \in [-1, 1])
\end{equation}
we obtain
\begin{equation*}
	\arccos{\left(-\frac{a}{b} \right)} - \abs{\theta}
	= \abs{\Arg(w)} - \arccos{\left( \frac{a}{\abs{w}} \right)}.
\end{equation*}
\item Case 2: $b < 0$.
Since $\abs{w} = -b$ and $\Arg(w) = \theta$, we have
\begin{equation*}
	\arccos{\left(-\frac{a}{b} \right)} - \abs{\theta}
	= - {\left[ \abs{\Arg(w)} - \arccos{\left( \frac{a}{\abs{w}} \right)} \right]}.
\end{equation*}
\end{itemize}
This completes the proof.
\end{proof}

\subsection{Critical delay and its domain of definition for real \texorpdfstring{$a$}{a} and \texorpdfstring{$w$}{w}}

The following result is a direct consequence of Theorem~\ref{thm: T(a, w), real a}.
Therefore, the proof can be omitted.
See \cite[Theorem 2.3]{Boese 1989} for the result with $a = 1$ and $w \le 0$.
See \cite[Theorem]{Boese 1993} for the result with general $a, w \in \R$.

\begin{theorem}[\cite{Boese 1993}, cf.\ \cite{Cooke--Grossman 1982}, \cite{Boese 1989}]\label{thm: T(a, w), real a and w}
Suppose $a, w \in \R$.
Then the following statements hold:
\begin{enumerate}[label=\textup{(\Roman*)}]
\item $T(a, w) = (0, \infty)$ if and only if $a > 0$ and $-a \le w < a$.
\item $T(a, w)$ is a nonempty proper subset of $(0, \infty)$ if and only if
\begin{equation*}
	w < -{\abs{a}}.
\end{equation*}
In this case, $T(a, w) = (0, \tau_\mathrm{c}(a, w))$ holds, where $\tau_\mathrm{c}(a, w) > 0$ is expressed by
\begin{equation*}
	\tau_\mathrm{c}(a, w)
	= \frac{1}{\sqrt{w^2 - a^2}} \arccos{\left( \frac{a}{w} \right)}
	= \frac{1}{\sqrt{w^2 - a^2}} \arccot{\left( -\frac{a}{\sqrt{w^2 - a^2}} \right)}.
\end{equation*}
\item $T(a, w)$ is empty if and only if $w \ge a$.
\end{enumerate}
\end{theorem}

See Fig.~\ref{fig: decomposition of (a, w)-plane by critical delay} for the picture of $(a, w)$-plane decomposed by the nature of $T(a, w)$ given in Theorem~\ref{thm: T(a, w), real a and w}.
It is a special case of Fig.~\ref{fig: decomposition of (a, |w|)-plane by critical delay}.
The above subset
\begin{equation*}
	\Set{(a, w) \in \R^2}{w < -{\abs{a}}}
\end{equation*}
is the domain of definition of critical delay function~\eqref{eq: critical delay function} $(a, w) \mapsto \tau_\mathrm{c}(a, w) \in (0, \infty)$ when $w$ varies in the real number line.
We note that the expressions of $\tau_\mathrm{c}(a, w)$ above are obtained by identities~\eqref{eq: arccos(x) + arccos(-x) = pi} and 
\begin{equation}\label{eq: arccot(x) + arccot(-x) = pi}
	\arccot(x) + \arccot(-x) = \pi
	\mspace{20mu}
	(x \in \R)
\end{equation}
since $w < 0$.

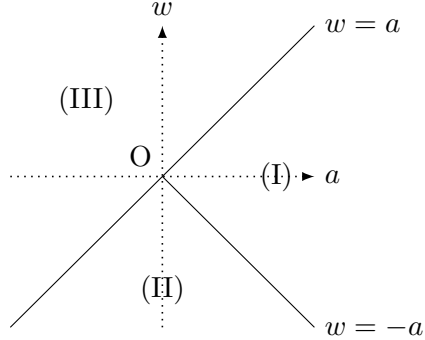
\begin{figure}[tbp]
\centering
\begin{tikzpicture}
	\draw[-latex, semithick, dotted] (-2,0) -- (2,0) node[right]{$a$}; 
	\draw[-latex, semithick, dotted] (0,-2) -- (0,2) node[above]{$w$}; 
	\draw (0,0) node[above left]{O}; 
	\draw[domain=-2:2, variable=\a] plot(\a,\a) node[right]{$w = a$};
	\draw[domain=0:2, variable=\a] plot(\a,-\a) node[right]{$w = -a$};
	\draw (1.5,0) node{(I)};
	\draw (0,-1.5) node{(II)};
	\draw (-1,1) node{(III)};
\end{tikzpicture}
\caption{Decomposition of $(a, w)$-plane obtained by Theorem~\ref{thm: T(a, w), real a and w}}
\label{fig: decomposition of (a, w)-plane by critical delay}
\end{figure}

\begin{remark}
When $a > w$ and Eq.~\eqref{eq: TE} is not asymptotically stable independent of delay, Cooke and Grossman~\cite[Section 2]{Cooke--Grossman 1982} discussed the existence of the critical delay value $\tau_0 > 0$ satisfying the following properties: (1) For all $0 < \tau < \tau_0$, all the roots have negative real parts, (2) For all $\tau > \tau_0$, there exists a root whose real part is positive.
However, the explicit expression of $\tau_0$ is not given in \cite{Cooke--Grossman 1982}.
\end{remark}

\subsection{Method by critical delay and stability region}

For the sake of clarity, we introduce the following notation.

\begin{notation}
Let
\begin{equation*}
	D_\mathrm{c}
	\coloneqq \Set{(a, w) \in \R \times (\C \setminus \{0\})}{\Re(w) < a < \abs{w}},
\end{equation*}
be a linear cone which is the domain of definition of critical delay function~\eqref{eq: critical delay function}.
It corresponds to the delay-dependent stability or instability region (II).
\end{notation}

Here a nonempty subset $C$ (of a linear topological space over $\R$) is called a \textit{linear cone} if for every $v \in C$ and every $s > 0$, $sv \in C$ holds.

Theorem~\ref{thm: T(a, w), real a} shows that all the roots of Eq.~\eqref{eq: TE} have negative real parts if and only if $a$ and $w$ satisfy one of the following properties:
\begin{enumerate}[label=\textup{(\Roman*)}]
\item $a \ge \abs{w}$ and $a > \Re(w)$.
\item $(a, w) \in D_\mathrm{c}$ and inequality~\eqref{eq: inequality on critical delay} $\tau_\mathrm{c}(a, w) > \tau$ holds.
\end{enumerate}

From expressions~\eqref{eq: tau_c(a, w), |Arg(w)|, arccos} or \eqref{eq: tau_c(a, w), |Arg(w)|, arccot} of the critical delay, we have
\begin{equation}\label{eq: tau_c(a tau, w tau)}
	\tau_\mathrm{c}(sa, sw) = \frac{1}{s}{\tau_\mathrm{c}(a, w)}
\end{equation}
for every $s > 0$ and every $(a, w) \in D_\mathrm{c}$.
Therefore, for each fixed $\tau > 0$, inequality~\eqref{eq: inequality on critical delay} is equivalent to $\tau_\mathrm{c}(\tau a, \tau w) > 1$.
This means that all the roots of Eq.~\eqref{eq: TE} have negative real parts if and only if $a$ and $w$ satisfy one of the following properties:
\begin{enumerate}[label=\textup{(\Roman*)}]
\item $a \ge \abs{w}$ and $a > \Re(w)$.
\item $(a, w) \in D_\mathrm{c}$ and an inequality $\tau_\mathrm{c}(\tau a, \tau w) > 1$ holds.
\end{enumerate}
We call the above method to find the stability region the \textit{method by critical delay}.

\section{Dependence of roots with respect to delay parameter}\label{sec: dependence of roots on delay parameter}

Let $a \in \C$ and $w \in \C \setminus \{0\}$ be given.
In this section, we study the dependence of zeros of the function $h(\argdot; a, w, s) \colon \C \to \C$ defined by
\begin{equation}\label{eq: h(z; a, w, s)}
	h(z; a, w, s)
	\coloneqq
	z + a - w\e^{-sz}
\end{equation}
on the real parameter $s \in \R$.

\subsection{Argument by Rouch\'{e}'s theorem}

For each fixed $s_0 \in \R$, we have
\begin{align*}
	\abs{h(z; a, w, s) - h(z; a, w, s_0)}
	&= \abs{w}\abs{\e^{-sz} - \e^{-s_0z}} \\
	&= \abs{w}\abs{\e^{-s_0z}}\abs{\e^{-(s - s_0)z} - 1}
\end{align*}
for all $z \in \C$ and all $s \in \R$.
Let $D$ be a given open disk of the complex plane $\C$ so that $h(\argdot; a, w, s_0)$ does not have any zeros on the boundary $\partial D$ of $D$.
We note that such an open disk can be always chosen because $h(\argdot; a, w, s_0)$ is not constant.
We have
\begin{equation*}
	\sup_{z \in \partial D} \abs{\e^{-s_0z}} \le M
\end{equation*}
for some $M > 0$.
Furthermore, the term $\abs{\e^{-(s - s_0)z} - 1}$ converges to $0$ as $s \to s_0$ uniformly in $z \in \partial D$.
This shows that
\begin{equation*}
	\sup_{z \in \partial D} \abs{h(z; a, w, s) - h(z; a, w, s_0)}
	< \inf_{z \in \partial D} \abs{h(z; a, w, s_0)}
\end{equation*}
holds for all sufficiently small $\abs{s - s_0}$ because $h(\argdot; a, w, s_0)$ has no zeros on $\partial D$.
Therefore, by applying Rouch\'{e}'s theorem, there exists a $\delta > 0$ such that for all $s \in (s_0 - \delta, s_0 + \delta)$, $h(\argdot; a, w, s)$ and $h(\argdot; a, w, s_0)$ have the same number of zeros on $D$ counted with multiplicity.

By combining the above argument and Theorem~\ref{thm: T(a, w), real a}, the following lemma is obtained.

\begin{lemma}
Let $a \in \R$ and $w \in \C \setminus \{0\}$ be fixed so that $(a, w) \in D_\mathrm{c}$.
Then Eq.~\eqref{eq: TE} with $\tau = \tau_\mathrm{c}(a, w)$ does not have any roots with positive real parts and has a root on the imaginary axis.
\end{lemma}

\begin{proof}
By combining the above argument and Theorem~\ref{thm: T(a, w), real a}, it holds that Eq.~\eqref{eq: TE} with $\tau = \tau_\mathrm{c}(a, w)$ does not have any roots with positive real parts.
To show the existence of a root on the imaginary axis, we use the following estimate: For every zero $z$ of $h(\argdot; a, w, s)$,
\begin{equation}\label{eq: modulus of zeros of h(argdot; a, w, s)}
	\abs{z}
	= \abs{-a + w\e^{-sz}}
	\le \abs{a} + {\abs{w}}\e^{-s\Re(z)}.
\end{equation}
We now suppose $s$ is nonnegative.
From \eqref{eq: modulus of zeros of h(argdot; a, w, s)}, all the zeros of $h(\argdot; a, w, s) = 0$ whose real parts are greater than or equal to $\sigma \in \R$ should be located in the closed disk with center at $0$ and radius
\begin{equation*}
	\abs{a} + {\abs{w}}\e^{-s\sigma}.
\end{equation*}
This shows that the number of such zeros is finite counted with multiplicity.
Therefore, the combination of the argument by Rouch\'{e}'s theorem and Theorem~\ref{thm: T(a, w), real a} also shows that Eq.~\eqref{eq: TE} with $\tau = \tau_\mathrm{c}(a, w)$ has a root on the imaginary axis.
\end{proof}

\subsection{Asymptotics of real parts of roots as delay tends to zero}\label{subsec: asymptotics of real parts of roots}

Inequality~\eqref{eq: modulus of zeros of h(argdot; a, w, s)} also holds when $a$ is complex and shows that the real part $\Re(z)$ satisfies
\begin{equation}\label{eq: real parts of zeros of h(argdot; a, w, s)}
	\abs{\Re(z)}
	\le \abs{a} + {\abs{w}}\e^{-s\Re(z)}
\end{equation}
because of $\abs{\Re(z)} \le \abs{z}$.
To analyze inequality~\eqref{eq: real parts of zeros of h(argdot; a, w, s)}, we will use the real branches of the Lambert $W$ function.
Here the \textit{Lambert $W$ function} is the inverse function of a complex function $\C \ni z \mapsto z\e^z \in \C$ in the sense of set-valued function, i.e., $W(\zeta)$ is defined by
\begin{equation*}
	W(\zeta) \coloneqq \Set{z \in \C}{z\e^z = \zeta}
\end{equation*}
for any $\zeta \in \C$.
We refer the reader to Corless et al.~\cite{Corless--Gonnet--Hare--Jeffrey--Knuth 1996} as a review paper on the Lambert $W$ function.
See \cite{Brito--Fabiao--Staubyn 2008} for a short survey of the Lambert $W$ function.
For a discussion on the naming of the Lambert $W$ function, e.g., see \cite{Hayes 2005}.
There is an attempt to generalize the Lambert $W$ function (see \cite{Mexo--Baricz 2017}).

The Lambert $W$ function is strongly related to Eq.~\eqref{eq: TE}.
Indeed, the set of all roots of Eq.~\eqref{eq: TE} is equal to
\begin{equation}\label{eq: set of all roots of z + a - we^{-tau z} = 0}
	\frac{1}{\tau} W(\tau w\e^{\tau a}) - a
	\coloneqq \Set{z - a}{z \in \frac{1}{\tau}W(\tau w\e^{\tau a})}
\end{equation}
because Eq.~\eqref{eq: TE} can be transformed into
\begin{equation*}
	\tau(z + a) \cdot \e^{\tau(z + a)} = \tau w\e^{\tau a}.
\end{equation*}
We note that it is common to use the Lambert $W$ function for numerical investigations of the location of the roots of Eq.~\eqref{eq: TE} (e.g., see \cite{Asl--Ulsoy 2003}, \cite{Hovel--Scholl 2005}, \cite{Hwang--Cheng 2005}, \cite{Shinozaki--Mori 2006}, \cite{Fiedler--Flunkert--Georgi--Hovel--Scholl 2007}, \cite{Wang--Hu 2008}, \cite{Yi--Yu--Kim 2011}, \cite{Bortz 2016}, and \cite{Surya--Vyasarayani--Kalmar-Nagy 2018}).

As is mentioned above, we use the real branches of the Lambert $W$ function for the study of the asymptotics of real parts of roots of Eq.~\eqref{eq: TE} as $\tau \downarrow 0$.
Since the function
\begin{equation*}
	\R \ni x \mapsto x\e^x \in \R
\end{equation*}
is strictly monotonically increasing on $[-1, \infty)$ and strictly monotonically decreasing on $(-\infty, -1]$, the Lambert $W$ function has two real branches $W_0 \colon [-1/\e, \infty) \to \R$ and $W_{-1} \colon [-1/\e, 0) \to \R$.
These graphs are depicted in Fig.~\ref{fig: real branches of Lambert W function}.

\begin{figure}[tbp]
\centering
\begin{tikzpicture}[smooth, samples=100]
	\draw[-latex, semithick] (-2,0) -- (3,0) node[right]{$y$}; 
	\draw[-latex, semithick] (0,-3) -- (0,2) node[above]{$x$}; 
	\draw (0,0) node[below right]{O}; 
	\draw (-1/e,-1) node[left]{$(-\frac{1}{\e},-1)$};
	\filldraw (-1/e,-1) circle [radius=1.2pt];
	\draw[domain=-1:1.05] plot({\x*exp(\x)},\x) node[above]{$x = W_0(y)$}; 
	\draw[domain=-1:-3, densely dashed] plot({\x*exp(\x)},\x) node[below]{$x = W_{-1}(y)$}; 
\end{tikzpicture}
\caption{Real branches $W_0$ (solid) and $W_{-1}$ (dashed) of the Lambert $W$ function}
\label{fig: real branches of Lambert W function}
\end{figure}
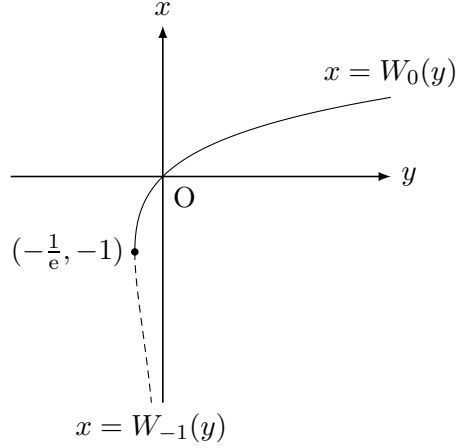

The following is the analytic result showing the asymptotics of real parts of roots.

\begin{theorem}
Let $a \in \C$ and $w \in \C \setminus \{0\}$ be fixed.
For each sufficiently small $\tau > 0$, let $\Sigma_0^+(\tau)$, $\Sigma_0^-(\tau)$, and $\Sigma_{-1}(\tau)$ be positive numbers given by
\begin{align*}
	\Sigma_0^+(\tau)
	&\coloneqq \frac{1}{\tau} W_0{\bigl( \tau{\abs{w}}\e^{-\tau\abs{a}} \bigr)} + \abs{a}, \\
	\Sigma_0^-(\tau)
	&\coloneqq \frac{1}{\tau} W_0{\bigl( -\tau{\abs{w}}\e^{\tau\abs{a}} \bigr)} - \abs{a},
\intertext{and}
	\Sigma_{-1}(\tau)
	&\coloneqq \frac{1}{\tau} W_{-1}{\bigl( -\tau{\abs{w}}\e^{\tau\abs{a}} \bigr)} - \abs{a}.
\end{align*}
Then for every sufficiently small $\tau > 0$, any root of Eq.~\eqref{eq: TE} satisfies
\begin{equation*}
	\Re(z) \le \Sigma_{-1}(\tau)
	\mspace{10mu} \text{or} \mspace{10mu}
	\Sigma_0^-(\tau) \le \Re(z) \le \Sigma_0^+(\tau),
\end{equation*}
where
\begin{equation*}
	\lim_{\tau \downarrow 0} \Sigma_0^\pm(\tau)
	= \pm(\abs{a} + \abs{w})
	\amd
	\lim_{\tau \downarrow 0} \Sigma_{-1}(\tau)
	= -\infty
\end{equation*}
hold.
\end{theorem}

\begin{proof}
Let $s$ be positive.
We transform inequality~\eqref{eq: real parts of zeros of h(argdot; a, w, s)} as follows.

\begin{itemize}
\item On inequality $\Re(z) - \abs{a} \le {\abs{w}}\e^{-s\Re(z)}$:
By multiplying both sides of the inequality by $s\e^{s(\Re(z) - \abs{a})} > 0$, it becomes
\begin{equation*}
	s(\Re(z) - \abs{a})\e^{s(\Re(z) - \abs{a})}
	\le s{\abs{w}}\e^{-s\abs{a}}.
\end{equation*}
Since the right-hand side is positive, the above inequality can be solved as
\begin{equation*}
	s(\Re(z) - \abs{a}) \le W_0{\bigl( s{\abs{w}}\e^{-s\abs{a}} \bigr)},
\end{equation*}
i.e., $\Re(z) \le \Sigma_0^+(s)$.
\item On inequality $\Re(z) + \abs{a} \ge -{\abs{w}}\e^{-s\Re(z)}$:
By multiplying both sides of the inequality by $s\e^{s(\Re(z) + \abs{a})} > 0$, it becomes
\begin{equation*}
	s(\Re(z) + \abs{a})\e^{s(\Re(z) + \abs{a})}
	\ge -s{\abs{w}}\e^{s\abs{a}}.
\end{equation*}
We may assume
\begin{equation*}
	0 < s{\abs{w}}\e^{s\abs{a}} < \frac{1}{\e}
\end{equation*}
by choosing $s > 0$ sufficiently small\footnote{By using $W_0$, the condition can be expressed as
\begin{equation*}
	s < \frac{1}{\abs{a}}W_0{\left( \frac{\abs{a}}{{\abs{w}}\e} \right)}
\end{equation*}
when $a \ne 0$.}.
Since $-s{\abs{w}}\e^{s\abs{a}} \in (-1/\e, 0)$, the above inequality can be solved as
\begin{align*}
	s(\Re(z) + \abs{a})
	&\ge W_0{\bigl( -s{\abs{w}}\e^{s\abs{a}} \bigr)}, \\
\intertext{or}
	s(\Re(z) + \abs{a})
	&\le W_{-1}{\bigl( -s{\abs{w}}\e^{s\abs{a}} \bigr)},
\end{align*}
i.e., $\Re(z) \ge \Sigma_0^-(s)$ or $\Re(z) \le \Sigma_{-1}(s)$.
\end{itemize}

By combining the above inequalities, we obtain the condition on $\Re(z)$.
The limit relations follow by $W_0'(0) = 1$ and $\lim_{y \uparrow 0} W_{-1}(y) = -\infty$.
This completes the proof.
\end{proof}

\subsection{Crossing direction of a purely imaginary root}\label{subsec: crossing direction of a purely imaginary root}

We investigate the crossing direction of a purely imaginary zero of $h(\argdot; a, w, s)$ with respect to the real parameter $s \in \R$.
It has been considered by Cooke and Grossman~\cite[Section 2]{Cooke--Grossman 1982} (for real $a, w$ case) and Matsunaga~\cite[Lemma 3]{Matsunaga 2008} (for real $a$ and complex $w$ case).
We note that $h(0; a, w, s) = 0$ for some $s$ if and only if $a = w$.
In this case,
\begin{equation*}
	h(0; a, a, s) = 0
\end{equation*}
for all $s \in \R$, and therefore, only the trivial branch is obtained as a function of $s$.
This is a consequence of the argument by Rouch\'{e}'s theorem.

Suppose $h(z_0; a, w, s_0) = 0$.
Since the derivative of $h(\argdot; a, w, s)$ is given by
\begin{equation*}
	h'(z; a, w, s) = 1 + sw\e^{-sz},
\end{equation*}
we have $h'(z_0; a, w, s_0) = 1 + s_0(z_0 + a)$.

\begin{theorem}[cf.\ \cite{Cooke--Grossman 1982}, \cite{Matsunaga 2008}]\label{thm: derivative of purely imaginary root, imaginary a}
Let $a \in \C$ and $w \in \C \setminus \{0\}$ be given.
Suppose $h(\iu \Omega_0; a, w, s_0) = 0$ holds for some $\Omega_0 \in \R \setminus \{0\}$ and $s_0 \in \R$.
We assume that $h'(\iu \Omega_0; a, w, s_0) \ne 0$, i.e.,
\begin{equation*}
	\iu \Omega_0 + a \ne -\frac{1}{s_0}
\end{equation*}
holds.
Then there exist a $\delta > 0$ and a unique continuously differentiable function
\begin{equation*}
	z(\argdot) \colon (s_0 - \delta, s_0 + \delta) \to \C
\end{equation*}
such that $z(s_0) = \iu \Omega_0$ and $h(z(s); a, w, s) = 0$ holds for all $\abs{s - s_0} < \delta$.
The derivative of $z(\argdot)$ at $s_0$ is calculated as
\begin{equation*}
	z'(s_0)
	= \frac{\Omega_0(\Omega_0 + \Im(a))}{(1 + s_0\Re(a))^2 + [s_0(\Omega_0 + \Im(a))]^2} - \iu \Omega_0\frac{s_0{\left[ (\Omega_0 + \Im(a))^2 + \Re(a)^2 \right]} + \Re(a)}{(1 + s_0\Re(a))^2 + [s_0(\Omega_0 + \Im(a))]^2}.
\end{equation*}
\end{theorem}

\begin{proof}
The existence of the function $z(\argdot)$ is a consequence of the implicit function theorem.
Therefore, we only need to calculate the derivative $z'(s_0)$.
Since
\begin{equation*}
	\frac{\partial h}{\partial z} = 1 + sw\e^{-sz}
	\amd
	\frac{\partial h}{\partial s} = zw\e^{-sz},
\end{equation*}
$z(\argdot)$ satisfies a differential equation
\begin{equation*}
	z'(s)
	= -\frac{z(s)(z(s) + a)}{1 + s(z(s) + a)}
\end{equation*}
by the implicit differentiation.
Here $w\e^{-sz(s)} = z(s) + a$ is used.
We note that the right-hand side of the above differential equation does not depend on $w$ explicitly.
By letting $s \coloneqq s_0$, we have
\begin{equation*}
	z'(s_0)
	= \frac{\Omega_0(\Omega_0 + \Im(a)) - \iu \Omega_0\Re(a)}{1 + s_0\Re(a) + \iu s_0 (\Omega_0 + \Im(a))}.
\end{equation*}
We omit the further calculation.
\end{proof}

\begin{remark}
When $s_0 = 0$, $h'(z_0; a, w, s_0) = 1$ holds.
When $s_0 \ne 0$, the condition $h'(z_0; a, w, s_0) \ne 0$ is equivalent to
\begin{equation*}
	z_0 \ne -\frac{1}{s_0} - a,
\end{equation*}
which is automatically satisfied when $z_0$ is purely imaginary and $a$ is real.
\end{remark}

\begin{remark}
Theorem~\ref{thm: derivative of purely imaginary root, imaginary a} shows that the crossing direction of purely imaginary zero with respect to the real parameter $s$ is determined by the sign of the real number
\begin{equation*}
	\Omega_0(\Omega_0 + \Im(a)).
\end{equation*}
It makes clear the difference between the real $a$ case and the imaginary $a$ case.
This should be compared with \cite[Proposition 1]{Cooke and Driessche 1986} and \cite[Section 3]{Boese 1998}.
\end{remark}

\section{Purely imaginary roots}\label{sec: purely imaginary roots}

Let $a \in \R$ and $w \in \C \setminus \{0\}$.
In this section, we find the $\tau$-values for which Eq.~\eqref{eq: TE} has a purely imaginary roots.

\subsection{Angular frequency equations}

For some $\Omega \in \R$, $\iu\Omega$ is a root of Eq.~\eqref{eq: TE} if and only if $\iu\Omega + a - w \e^{-\iu\tau\Omega} = 0$, i.e.,
\begin{equation}\label{eq: angular frequency}
	\left\{
	\begin{aligned}
		a - \abs{w}\cos(\Arg(w) - \tau\Omega) &= 0, \\
		\Omega - \abs{w}\sin(\Arg(w) - \tau\Omega) &= 0.
	\end{aligned}
	\right.
\end{equation}
When $\Omega$ is considered to be an unknown variable, we call \eqref{eq: angular frequency} the \textit{angular frequency equations}.
In this consideration, it is natural to assume that the delay parameter $\tau$ is also one of the unknown variables.
Then it is expected that one can solve \eqref{eq: angular frequency} with respect to $(\Omega, \tau)$ for each given $(a, w) = (a, {\abs{w}}, \Arg(w))$.

From angular frequency eqs.~\eqref{eq: angular frequency}, $\Omega$ necessarily satisfies
\begin{equation*}
	\abs{w}^2 = a^2 + \Omega^2
\end{equation*}
by the trigonometric identity $\cos^2(\argdot) + \sin^2(\argdot) \equiv 1$.
This also imposes that $\abs{w}^2 - a^2 \ge 0$, i.e.,
\begin{equation*}
	\abs{a} \le \abs{w} \amd \Omega = \pm\sqrt{\abs{w}^2 - a^2}.
\end{equation*}
When $\abs{a} = \abs{w}$, the following statements hold:
\begin{itemize}
\item Suppose $a = \abs{w}$.
Then Eq.~\eqref{eq: TE} has a root on the imaginary axis if and only if $\Arg(w) = 0$.
\item Suppose $a = -{\abs{w}}$.
Then Eq.~\eqref{eq: TE} has a root on the imaginary axis if and only if $\Arg(w) = \pi$.
\end{itemize}
In the above cases, $0$ is the only root on the imaginary axis.
This consideration shows that nontrivial situations will occur only when $w \ne 0$ and $\abs{a} < \abs{w}$.
We note that $D_\mathrm{c}$ is contained in this region, and the delay-independent instability region (III) intersects with this region.

\subsection{Angular frequencies}

To study nontrivial situations, we introduce the following notation.

\begin{notation}
Suppose $a \in \R$, $w \in \C \setminus \{0\}$, and $\abs{a} < \abs{w}$.
Let
\begin{equation*}
	\Omega(a, w)
	\coloneqq \Omega(a, \abs{w})
	\coloneqq \sqrt{\abs{w}^2 - a^2}.
\end{equation*}
Here the expression of $\Omega(a, w)$ does not depend on the argument $\Arg(w)$ of $w$.
\end{notation}

By using this notation, the critical delay $\tau_\mathrm{c}(a, w)$ is expressed by
\begin{align*}
	\tau_\mathrm{c}(a, w)
	&= \tau_\mathrm{c}(a, {\abs{w}}, \Arg(w)) \\
	&\coloneqq \frac{1}{\Omega(a, \abs{w})} \left[ \abs{\Arg(w)} - \arccos{\left( \frac{a}{\abs{w}} \right)} \right]
\end{align*}
for every $(a, w) \in D_\mathrm{c}$.
We note that $\tau_\mathrm{c}(a, w)$ depends on the absolute value of the argument of $w$.

The following lemmas give equivalent forms to angular frequency eqs.~\eqref{eq: angular frequency} under the additional conditions of
\begin{equation*}
	\Arg(w) - \tau\Omega \in [0, \pi] + 2k\pi
	\mspace{10mu} \text{or} \mspace{10mu}
	\Arg(w) - \tau\Omega \in (-\pi, 0) + 2k\pi
\end{equation*}
for some $k \in \Z$.
We note that such an integer $k$ uniquely exists by the decomposition
\begin{equation}\label{eq: decomposition of real number line}
	\R = \bigcup_{k \in \Z} ((-\pi, \pi] + 2k\pi)
\end{equation}
of the real number line.

\begin{lemma}\label{lem: angular frequency, [0, pi]}
Suppose $a \in \R$, $w \in \C \setminus \{0\}$, and $\abs{a} < \abs{w}$.
Let $k \in \Z$ be given.
Then $\Omega \in \R \setminus \{0\}$ satisfies angular frequency eqs.~\eqref{eq: angular frequency} and $\Arg(w) - \tau\Omega \in [0, \pi] + 2k\pi$ if and only if
\begin{equation*}
	\Omega = \Omega(a, w)
	\amd
	\tau\Omega = \Arg(w) - \arccos{\left( \frac{a}{\abs{w}} \right)} - 2k\pi
\end{equation*}
hold.
\end{lemma}

\begin{proof}
(Only-if-part). By the assumption, we have
\begin{equation*}
	\Arg(w) - \tau\Omega - 2k\pi = \arccos{\left( \frac{a}{\abs{w}} \right)}.
\end{equation*}
Therefore,
\begin{equation*}
	\Omega
	= \abs{w}\sin{\left( \arccos{\left( \frac{a}{\abs{w}} \right)} \right)}
	= \Omega(a, w)
\end{equation*}
holds.

(If-part). We only need to check whether the second equation of Eqs.~\eqref{eq: angular frequency} holds.
This is verified in view of
\begin{equation*}
	\abs{w}\sin(\Arg(w) - \tau\Omega)
	= \abs{w}\sin{\left( \arccos{\left( \frac{a}{\abs{w}} \right)} \right)}
	= \Omega(a, w)
	= \Omega.
\end{equation*}

This completes the proof.
\end{proof}

In the similar way, the following lemma is obtained.
The proof can be omitted.

\begin{lemma}\label{lem: angular frequency, (-pi, 0)}
Suppose $a \in \R$, $w \in \C \setminus \{0\}$, and $\abs{a} < \abs{w}$.
Let $k \in \Z$ be given.
Then $\Omega \in \R \setminus \{0\}$ satisfies angular frequency eqs.~\eqref{eq: angular frequency} and $\Arg(w) - \tau\Omega \in (-\pi, 0) + 2k\pi$ if and only if
\begin{equation*}
	\Omega = -\Omega(a, w)
	\amd
	\tau\Omega = \Arg(w) + \arccos{\left( \frac{a}{\abs{w}} \right)} - 2k\pi
\end{equation*}
hold.
\end{lemma}

\subsection{Sequence of $\tau$-values}

In view of Lemmas~\ref{lem: angular frequency, [0, pi]} and \ref{lem: angular frequency, (-pi, 0)}, we introduce the following notation.

\begin{notation}
Suppose $a \in \R$, $w \in \C \setminus \{0\}$, and $\abs{a} < \abs{w}$.
For each integer $n \ge 1$, let
\begin{align*}
	\tau_n^\pm(a, w)
	&\coloneqq \tau_n^\pm(a, {\abs{w}}, \Arg(w)) \\
	&\coloneqq \frac{1}{\Omega(a, w)} \left[ \pm \Arg(w) - \arccos{\left( \frac{a}{\abs{w}} \right)} + 2n\pi \right] \\
	&= \frac{1}{\Omega(a, w)} \left[ \pm \Arg(w) - \arccot{\left( \frac{a}{\Omega(a, w)} \right)} + 2n\pi \right],
\end{align*}
where $\tau_n^\pm(a, w)$ are always positive because of $\arccos(a/{\abs{w}}) < \pi$. 
\end{notation}

\begin{remark}
We have
\begin{equation}\label{eq: tau_n^+(a, w) - tau_n^-(a, w)}
	\tau_n^+(a, w) - \tau_n^-(a, w)
	= \frac{2\Arg(w)}{\Omega(a, w)}
\end{equation}
for all $n \ge 1$.
\end{remark}

The following theorem\footnote{It is sufficient to consider the case $\Arg(w) \in [0, \pi]$ (i.e., the case $\Im(w) \ge 0$) because we obtain
\begin{equation*}
	\bar{z} + a - \bar{w}\e^{-\tau\bar{z}} = 0
\end{equation*}
by taking the complex conjugate.
However, we do not adopt this assumption because it does not bring us any simplification as the statement and the proof show.} gives conditions on $\tau$ under which Eq.~\eqref{eq: TE} has a purely imaginary root for each given $a$ and $w$.

\begin{theorem}[cf.\ \cite{Matsunaga 2008}]\label{thm: tau-sequence, angular frequency}
Suppose $a \in \R$, $w \in \C \setminus \{0\}$, and $\abs{a} < \abs{w}$.
Then Eq.~\eqref{eq: TE} has a root on the imaginary axis if and only if one of the following conditions is satisfied:
\begin{enumerate}[label=\textup{(\roman*)}]
\item $(a, w) \in D_\mathrm{c}$ and $\tau = \tau_\mathrm{c}(a, w)$.
\item $\tau = \tau_n^+(a, w)$ for some integer $n \ge 1$.
\item $\tau = \tau_n^-(a, w)$ for some integer $n \ge 1$.
\end{enumerate}
Furthermore, the following statements hold:
\begin{itemize}
\item When \textup{(i)} with $\Arg(w) > 0$ or \textup{(ii)} holds, Eq.~\eqref{eq: TE} has only the root $\iu\Omega(a, w)$ on the imaginary axis.
\item When \textup{(i)} with $\Arg(w) < 0$ or \textup{(iii)} holds, Eq.~\eqref{eq: TE} has only the root $-\iu\Omega(a, w)$ on the imaginary axis.
\end{itemize}
\end{theorem}

\begin{proof}
It holds that Eq.~\eqref{eq: TE} has a root on the imaginary axis if and only if there exists $\Omega \in \R \setminus \{0\}$ satisfying angular frequency eqs.~\eqref{eq: angular frequency}.
From~\eqref{eq: decomposition of real number line}, the consideration is divided into the following two cases.
\begin{itemize}[leftmargin=*]
\item Case 1: $\Arg(w) - \tau\Omega \in [0, \pi] + 2k\pi$ for some $k \in \Z$.
In this case, angular frequency eqs.~\eqref{eq: angular frequency} is reduced to
\begin{equation*}
	\Omega = \Omega(a, w)
	\amd
	\tau\Omega = \Arg(w) - \arccos{\left( \frac{a}{\abs{w}} \right)} - 2k\pi
\end{equation*}
from Lemma~\ref{lem: angular frequency, [0, pi]}.
The positivity of $\tau\Omega$ imposes the following conditions:
	\begin{itemize}
	\item $k = 0$, $\Arg(w) > 0$, and $a > \Re(w)$.
	In this case, $\tau = \tau_\mathrm{c}(a, w)$.
	\item $k \le -1$.
	In this case, $\tau = \tau_{-k}^+(a, w)$.
	\end{itemize}
\item Case 2: $\Arg(w) - \tau\Omega \in (-\pi, 0) + 2k\pi$ for some $k \in \Z$.
In this case, angular frequency eqs.~\eqref{eq: angular frequency} is reduced to
\begin{equation*}
	\Omega = -\Omega(a, w)
	\amd
	\tau\Omega = \Arg(w) + \arccos{\left( \frac{a}{\abs{w}} \right)} - 2k\pi
\end{equation*}
from Lemma~\ref{lem: angular frequency, (-pi, 0)}.
The negativity of $\tau\Omega$ imposes the following conditions:
	\begin{itemize}
	\item $k = 0$, $\Arg(w) < 0$, and $a > \Re(w)$.
	In this case, $\tau = \tau_\mathrm{c}(a, w)$.
	\item $k \ge 1$.
	In this case, $\tau = \tau_k^-(a, w)$.
	\end{itemize}
\end{itemize}
This completes the proof.
\end{proof}

For each given $(a, w)$ satisfying $\abs{a} < \abs{w}$, Theorem~\ref{thm: tau-sequence, angular frequency} gives conditions on $\tau$ for which Eq.~\eqref{eq: TE} has a purely imaginary root.
Here the corresponding angular frequency is uniquely determined as a function of $(a, w)$ and does not explicitly depends on $\tau$.

\begin{remark}
In the sense of asymptotic stability of the trivial solution, DDE~\eqref{eq: x'(t) = -ax(t) + wx(t - tau)} and the DDE
\begin{equation*}
	\dot{x}(t) = -\tau a + \tau w x(t - 1)
\end{equation*}
are equivalent.
However, the angular frequency of the latter equation explicitly depends on the parameter $\tau$ through the coefficients $\tau a$ and $\tau w$.
This is natural because the latter DDE is obtained by the change of time variables $t \to \tau t$ depending on the parameter $\tau$.
\end{remark}

\subsubsection{Comparison with Matsunaga's sequences}

Matsunaga~\cite[Lemmas 1 and 2]{Matsunaga 2008} has obtained the similar results, where the case $0 < \theta \le \pi/2$ is only considered and the proof is divided into the cases $b > 0$ (\cite[Lemma 1]{Matsunaga 2008}) and $b < 0$ (\cite[Lemma 2]{Matsunaga 2008}).
We recall that the parameter
\begin{equation*}
	w = -b\e^{\iu\theta} \mspace{20mu} (\text{$b \in \R \setminus \{0\}$, $\theta \in [-\pi/2, \pi/2]$})
\end{equation*}
is used in \cite{Matsunaga 2008}.
In these results, there are sequences $(\tau_n^\pm)_{n = 0}^\infty$ for the case $b > 0$ and $(\sigma_n^\pm)_{n = 0}^\infty$ for the case $b < 0$ for which Eq.~\eqref{eq: TE} has a nonzero root on the imaginary axis.
These are given by
\begin{equation*}
	\tau_n^\pm
	\coloneqq \frac{1}{\sqrt{b^2 - a^2}} \left[ \pm\theta + \arccos{\left( -\frac{a}{b} \right)} + 2n\pi \right],
\end{equation*}
and
\begin{align*}
	\sigma_n^+
	&\coloneqq \frac{1}{\sqrt{b^2 - a^2}} \left[ \theta - \arccos{\left( -\frac{a}{b} \right)} + 2n\pi \right], \\
	\sigma_n^-
	&\coloneqq \frac{1}{\sqrt{b^2 - a^2}} \left[ -\theta - \arccos{\left( -\frac{a}{b} \right)} + (2n + 2)\pi \right].
\end{align*}

It seems that there are two sequences for the existence of a nonzero root on the imaginary axis.
However, Theorem~\ref{thm: tau-sequence, angular frequency} shows that this is not the situation, which is clarified in the following lemmas.

\begin{lemma}
Let $w = -b\e^{\iu\theta}$ for $b \in \R \setminus \{0\}$ and $\theta \in [-\pi/2, \pi/2]$.
Suppose $b > 0$ and $\theta > 0$.
Then for all $n \ge 0$,
\begin{equation*}
	\tau_n^+ = \tau_{n + 1}^+(a, w)
	\amd
	\tau_n^- =
	\begin{cases}
		\tau_\mathrm{c}(a, w) & (\text{$n = 0$, $\Re(w) < a < \abs{w}$}), \\
		\tau_n^-(a, w) & (n \ge 1)
	\end{cases}
\end{equation*}
hold.
\end{lemma}

\begin{proof}
By the assumption, $\abs{w} = b$ and $\Arg(w) = \theta - \pi < 0$ from the proof of Lemma~\ref{lem: critical delay in Matsunaga}.
Therefore, we have
\begin{equation*}
	\theta + \arccos{\left( -\frac{a}{b} \right)}
	= (\Arg(w) + \pi) + \left[ \pi - \arccos{\left( \frac{a}{\abs{w}} \right)} \right]
\end{equation*}
and
\begin{equation*}
	-\theta + \arccos{\left( -\frac{a}{b} \right)}
	= -(\Arg(w) + \pi) + \left[ \pi - \arccos{\left( \frac{a}{\abs{w}} \right)} \right]
\end{equation*}
from identity~\eqref{eq: arccos(x) + arccos(-x) = pi}.
This shows the conclusion.
\end{proof}

\begin{lemma}
Let $w = -b\e^{\iu\theta}$ for $b \in \R \setminus \{0\}$ and $\theta \in [-\pi/2, \pi/2]$.
Suppose $b < 0$ and $\theta > 0$.
Then for all $n \ge 0$,
\begin{equation*}
	\sigma_n^+ =
	\begin{cases}
		\tau_\mathrm{c}(a, w) & (\text{$n = 0$, $\Re(w) < a < \abs{w}$}), \\
		\tau_n^+(a, w) & (n \ge 1)
	\end{cases}
	\amd
	\sigma_n^- = \tau_{n + 1}^-(a, w)
\end{equation*}
hold.
\end{lemma}

\begin{proof}
By the assumption, we have $\abs{w} = -b$ and $\Arg(w) = \theta > 0$.
Then the expressions are obtained by simply using these relations.
\end{proof}

\begin{remark}
The positivity of $\tau_0^-$ (for the case that $b > 0$ and $\theta \in (0, \pi/2]$) and $\sigma_0^+$ (for the case $b < 0$ and $\theta \in (0, \pi/2]$) must be checked because this is essential for the expression of $\tau_\mathrm{c}(a, w)$.
However, this has not been performed in \cite{Matsunaga 2008}.
\end{remark}

\subsection{Ordering of $\tau$-values}

Here we consider the ordering of $\tau_\mathrm{c}(a, w)$ and $\tau_n^\pm(a, w)$ for $n \ge 1$.
The cases $\Arg(w) = 0$ or $\Arg(w) = \pi$ are special.

\begin{lemma}\label{lem: ordering, real a and w}
Suppose $a \in \R$, $w \in \R \setminus \{0\}$, and $\abs{a} < \abs{w}$.
Then the following statements hold:
\begin{enumerate}[label=\textup{\arabic*.}]
\item If $w > 0$, then $\tau_\mathrm{c}(a, w)$ is not defined and we have
\begin{equation*}
	\tau_n^+(a, w)
	= \tau_n^-(a, w)
	= \frac{1}{\Omega(a, w)} \left[ - \arccos{\left( \frac{a}{w} \right)} + 2n\pi \right]
\end{equation*}
for all integer $n \ge 1$.
\item If $w < 0$, then $\tau_\mathrm{c}(a, w)$ is defined.
Furthermore, we have
\begin{equation*}
	\tau_\mathrm{c}(a, w)
	= \tau_1^-(a, w)
	= \frac{1}{\Omega(a, w)} \arccos{\left( \frac{a}{w} \right)}
\end{equation*}
and
\begin{equation*}
	\tau_n^+(a, w)
	= \tau_{n + 1}^-(a, w)
	= \frac{1}{\Omega(a, w)} \left[ \arccos{\left( \frac{a}{w} \right)} + 2n\pi \right]
\end{equation*}
for all integer $n \ge 1$.
\end{enumerate}
\end{lemma}

\begin{proof}
1. Since $\Arg(w) = 0$, Theorem~\ref{thm: T(a, w), real a and w} shows that $\tau_\mathrm{c}(a, w)$ is not defined.
The expressions of $\tau_n^\pm(a, w)$ are obvious.

2. Since $\Arg(w) = \pi$, $\tau_\mathrm{c}(a, w)$ is defined from Theorem~\ref{thm: T(a, w), real a and w}.
Furthermore, we have
\begin{equation*}
	\tau_1^-(a, w)
	= \tau_\mathrm{c}(a, w)
	= \frac{1}{\Omega(a, w)} \left[ \pi - \arccos{\left( \frac{a}{\abs{w}} \right)} \right]
\end{equation*}
and
\begin{equation*}
	\tau_{n + 1}^-(a, w)
	= \tau_n^+(a, w)
	= \frac{1}{\Omega(a, w)} \left[ \pi - \arccos{\left( \frac{a}{\abs{w}} \right)} + 2n\pi \right].
\end{equation*}
In view of identity~\eqref{eq: arccos(x) + arccos(-x) = pi}, the expressions are obtained.
\end{proof}

\begin{lemma}\label{lem: ordering, real a, 0 < Arg(w) < pi}
Suppose $a \in \R$, $w \in \C \setminus \{0\}$, $\abs{a} < \abs{w}$, and $\Arg(w) \in (0, \pi)$.
Then for all $n \ge 1$,
\begin{equation*}
	\tau_n^-(a, w) < \tau_n^+(a, w) < \tau_{n + 1}^-(a, w)
\end{equation*}
holds.
Furthermore, if $a > \Re(w)$, then
\begin{equation*}
	0 < \tau_\mathrm{c}(a, w) < \tau_1^-(a, w)
\end{equation*}
holds.
\end{lemma}

\begin{proof}
The inequality $\tau_n^-(a, w) < \tau_n^+(a, w)$ is a consequence of \eqref{eq: tau_n^+(a, w) - tau_n^-(a, w)}.
The inequality $\tau_n^+(a, w) < \tau_{n + 1}^-(a, w)$ follows by
\begin{equation*}
	\tau_{n + 1}^-(a, w) - \tau_n^+(a, w)
	= \frac{2(\pi - \Arg(w))}{\Omega(a, w)}
	> 0.
\end{equation*}
The inequality $\tau_\mathrm{c}(a, w) < \tau_1^-(a, w)$ for the case $a > \Re(w)$ follows by the same reasoning.
\end{proof}

\begin{lemma}\label{lem: ordering, real a, -pi < Arg(w) < 0}
Suppose $a \in \R$, $w \in \C \setminus \{0\}$, $\abs{a} < \abs{w}$, and $\Arg(w) \in (-\pi, 0)$.
Then for all $n \ge 1$,
\begin{equation*}
	\tau_n^+(a, w) < \tau_n^-(a, w) < \tau_{n + 1}^+(a, w)
\end{equation*}
holds.
Furthermore, if $a > \Re(w)$, then
\begin{equation*}
	0 < \tau_\mathrm{c}(a, w) < \tau_1^+(a, w)
\end{equation*}
holds.
\end{lemma}

\begin{proof}
The inequality $\tau_n^+(a, w) < \tau_n^-(a, w)$ is a consequence of \eqref{eq: tau_n^+(a, w) - tau_n^-(a, w)}.
The inequality $\tau_n^-(a, w) < \tau_{n + 1}^+(a, w)$ follows by
\begin{equation*}
	\tau_{n + 1}^+(a, w) - \tau_n^-(a, w)
	= \frac{2(\pi + \Arg(w))}{\Omega(a, w)}
	> 0.
\end{equation*}
The inequality $\tau_\mathrm{c}(a, w) < \tau_1^+(a, w)$ for the case $a > \Re(w)$ follows by the same reasoning.
\end{proof}

\section{Method by critical delay and stability region for real $a$ and $w$}\label{sec: method by critical delay for real a and w}

In this section, we study Eq.~\eqref{eq: TE} with real $a, w$ and find the stability region via the method by critical delay.
As is used in Introduction, we will use the following notation.

\begin{notation}
Let $C \colon (0, \pi) \to \R$ be the function defined by
\begin{equation*}
	C(\theta)
	\coloneqq \theta \cot\theta
	= \theta \cdot \frac{\cos\theta}{\sin\theta}
\end{equation*}
for all $\theta \in (0, \pi)$.
\end{notation}

Since
\begin{equation*}
	C'(\theta)
	= \frac{\cos\theta}{\sin\theta} - \frac{\theta}{\sin^2\theta}
	= \frac{\sin{2\theta} - 2\theta}{2\sin^2\theta}
	< 0
\end{equation*}
for $\theta \in (0, \pi)$ and
\begin{equation*}
	\lim_{\theta \downarrow 0} \frac{\theta}{\sin\theta} \cdot \cos\theta
	= 1,
\end{equation*}
the function $C \colon (0, \pi) \to \R$ is strictly monotonically decreasing and satisfies
\begin{equation*}
	\lim_{\theta \downarrow 0} C(\theta) = 1 \amd \lim_{\theta \uparrow \pi} C(\theta) = -\infty.
\end{equation*}
See Fig.~\ref{fig: theta cot(theta)} for the graph of the function $C \colon (0, \pi) \to \R$.

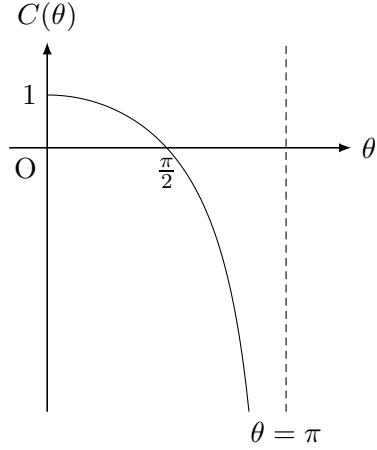
\begin{figure}[tbp]
\centering
\begin{tikzpicture}[smooth, samples=100, variable=\theta, yscale=0.7]
	\draw[-latex, semithick] (-0.5,0) -- (4,0) node[right]{$\theta$}; 
	\draw[-latex, semithick] (0,-5) -- (0,2) node[above]{$C(\theta)$};
	\draw (0,0) node[below left]{O}; 
	\draw[domain=0.001:2.655] plot(\theta,{\theta*cos(\theta r)/sin(\theta r)});
	\draw[densely dashed] (pi,-5) node[below]{$\theta = \pi$} -- (pi,2);
	\draw (0,1) node[left]{$1$};
	\draw (pi/2,0) node[below]{$\frac{\pi}{2}$};
\end{tikzpicture}
\caption{Graph of $C(\theta) = \theta \cot\theta$}
\label{fig: theta cot(theta)}
\end{figure}

The above properties show that the function $C$ gives a one-to-one correspondence between open intervals $(0, \pi)$ and $(-\infty, 1)$.
Therefore, the inverse function $C^{-1} \colon (-\infty, 1) \to \R$ is strictly monotonically decreasing and satisfies
\begin{equation}\label{eq: limits of Theta(r)}
	\lim_{r \to -\infty} C^{-1}(r) = \pi
	\amd
	\lim_{r \uparrow 1} C^{-1}(r) = 0.
\end{equation}

\subsection{Inequality on critical delay and stability region}

The following function will naturally appear for solving the inequality $\tau_\mathrm{c}(a, w) > 1$.

\begin{definition}
Let $R \colon (-\infty, 1) \to \R$ be the function defined by
\begin{equation*}
	R(r) \coloneqq \frac{C^{-1}(r)}{\sin C^{-1}(r)}
\end{equation*}
for all $r < 1$.
We also have
\begin{equation*}
	R(r) = \frac{r}{\cos C^{-1}(r)}
\end{equation*}
when $r \ne 0$ because $C^{-1}(r) \cot C^{-1}(r) = r$.
\end{definition}

We note that $R(r)$ can be considered as a function of $C^{-1}(r)$.
The following lemma gives qualitative properties of the function $R$.

\begin{lemma}\label{lem: property of R}
The function $R \colon (-\infty, 1) \to \R$ is strictly monotonically decreasing and satisfies $\lim_{r \to -\infty} R(r) = \infty$ and $\lim_{r \uparrow 1} R(r) = 1$.
Furthermore,
\begin{equation*}
	\lim_{r \to -\infty} \frac{R(r)}{\abs{r}} = 1
\end{equation*}
holds.
\end{lemma}

\begin{proof}
Since
\begin{equation*}
	\frac{\mathrm{d}}{\mathrm{d}\theta} \mspace{2mu} \frac{\theta}{\sin\theta}
	= \frac{\sin\theta - \theta \cos\theta}{\sin^2\theta}
	= \frac{1 - \theta \cot\theta}{\sin\theta}
	> 0
\end{equation*}
for $\theta \in (0, \pi)$, the function $(0, \pi) \ni \theta \mapsto \theta/\sin\theta \in \R$ is strictly monotonically increasing.
Therefore, it holds that the function $R$ is strictly monotonically decreasing because the function $C^{-1} \colon (-\infty, 1) \to \R$ is strictly monotonically decreasing.
The limits are consequences of \eqref{eq: limits of Theta(r)} and $R(r) = r/{\cos C^{-1}(r)}$ for $r \ne 0$.
\end{proof}

\begin{theorem}\label{thm: tau_c(a, w) > 1, real a and w}
Let $a, w \in \R$ be given so that $w < -{\abs{a}}$.
Then $\tau_\mathrm{c}(a, w) > 1$ if and only if
\begin{equation*}
	a > -1 \amd {-R(-a)} < w.
\end{equation*}
hold.
\end{theorem}

\begin{proof}
The inequality $\tau_\mathrm{c}(a, w) > 1$ is equivalent to
\begin{equation*}
	\sqrt{w^2 - a^2}
	< \arccot{\left( -\frac{a}{\sqrt{w^2 - a^2}} \right)}
\end{equation*}
by the expression of $\tau_\mathrm{c}(a, w)$.
Let $X(a, w) \coloneqq \sqrt{w^2 - a^2}$.
Since $\cot|_{(0, \pi)} \colon (0, \pi) \to \R$ is strictly monotonically decreasing, the above inequality can be solved as
\begin{equation*}
	a > -1 \amd X(a, w) < C^{-1}(-a).
\end{equation*}
By solving the last inequality with respect to $w$, we obtain
\begin{align*}
	w^2
	&< a^2 + C^{-1}(-a)^2 \\
	&= C^{-1}(-a)^2 \bigl( \cot^2{C^{-1}(-a)} + 1 \bigr),
\end{align*}
which is equivalent to $-R(-a) < w$ because of the negativity of $w$.
\end{proof}

From Theorem~\ref{thm: tau_c(a, w) > 1, real a and w}, the region in $(a, w)$-plane obtained by the inequality $\tau_\mathrm{c}(a, w) > 1$ is expressed by the function $R \colon (-\infty, 1) \to \R$, whose qualitative properties have been revealed by Lemma~\ref{lem: property of R}.

The following stability condition on $a$, $w$, and $\tau$ is obtained as a corollary of Theorems~\ref{thm: T(a, w), real a and w} and \ref{thm: tau_c(a, w) > 1, real a and w}.
The result is due to Hayes~\cite[Theorem 1]{Hayes 1950}.

\begin{corollary}[\cite{Hayes 1950}, refs.\ \cite{Hale--Lunel 1993}, \cite{Diekmann--vanGils--Lunel--Walther 1995}]\label{cor: TE, real a and w}
Suppose $a, w \in \R$.
Then all the roots of Eq.~\eqref{eq: TE} have negative real parts if and only if  \begin{equation*}
	a > -\frac{1}{\tau}
	\amd
	{-\frac{1}{\tau}}R(-\tau a) < w < a
\end{equation*}
hold.
\end{corollary}

\begin{proof}
From Theorem~\ref{thm: T(a, w), real a and w}, all the roots of Eq.~\eqref{eq: TE} have negative real parts if and only if one of the following conditions is satisfied:
\begin{itemize}
\item $a > 0$ and $-a \le w < a$.
\item $w < -{\abs{a}}$ and $\tau_\mathrm{c}(a, w) > \tau$.
\end{itemize}
Here the second condition becomes
\begin{equation*}
	\tau a > -1 \amd {-R(-\tau a)} < \tau w < -\tau{\abs{a}}
\end{equation*}
from equality~\eqref{eq: tau_c(a tau, w tau)} and Theorem~\ref{thm: tau_c(a, w) > 1, real a and w}.
By combining this and the first condition, we obtain the conclusion from Lemma~\ref{lem: property of R}.
\end{proof}

\begin{remark}\label{rmk: critical delay, real a and w}
In general, it is not apparent how to derive the expression of the critical delay from Corollary~\ref{cor: TE, real a and w}.
By following the argument of the proof of Theorem~\ref{thm: tau_c(a, w) > 1, real a and w} in reverse, we obtain the following equivalences: For $w < -{\abs{a}}$,
\begin{align*}
	&\tau a > -1 \amd {-\frac{1}{\tau}}R(-\tau a) < w \\
	&\iff \tau a > -1 \amd \tau \sqrt{w^2 - a^2} < C^{-1}(-\tau a) \\
	&\iff 0 < \tau \sqrt{w^2 - a^2} < \pi \amd \tau \sqrt{w^2 - a^2} \cot \Bigl( \tau\sqrt{w^2 - a^2} \Bigr) > -\tau a \\
	&\iff \tau \sqrt{w^2 - a^2} < \arccot{\left( -\frac{a}{\sqrt{w^2 - a^2}} \right)}.
\end{align*}
This shows
\begin{equation*}
	\tau_\mathrm{c}(a, w) = \frac{1}{\sqrt{w^2 - a^2}} \arccot{\left( -\frac{a}{\sqrt{w^2 - a^2}} \right)}.
\end{equation*}
\end{remark}

\subsection{Parametrization of stability boundary curve}\label{subsec: parametrization of boundary curve for real a, w}

Since the function $C^{-1} \colon (-\infty, 1) \to \R$ gives a one-to-one correspondence between the open intervals $(-\infty, 1)$ and $(0, \pi)$, the curve
\begin{equation*}
	\Set{\left( a, -\frac{1}{\tau}R(-\tau a) \right)}{a > -\frac{1}{\tau}}
\end{equation*}
in $(a, w)$-plane is parametrized by
\begin{equation}\label{eq: stability boundary curve for real a, w}
	a = -\frac{1}{\tau} \theta\cot{\theta}
	\amd
	w = -\frac{\theta}{\tau\sin{\theta}}\end{equation}
for $\theta \in (0, \pi)$ in view of
\begin{equation*}
	-\tau a = C^{-1}(-\tau a)\cot C^{-1}(-\tau a).
\end{equation*}
The stability boundary curves are depicted in Fig.~\ref{fig: TE, real a and w} for the cases of $\tau = 1$ and $\tau = 1/3$.
The picture is well-known in the literature (see \cite[Figure 5.1 in Chapter 5]{Hale--Lunel 1993} and \cite[Figure XI.1 in Chapter XI]{Diekmann--vanGils--Lunel--Walther 1995}).
See also \cite[Comment after Theorem A]{Matsunaga 2008}.

\begin{figure}[tbp]
\centering
\begin{tikzpicture}[smooth, samples=100, variable=\a, scale=0.5]
	\draw[-latex, semithick, dotted] (-6,0) -- (6,0) node[right]{$a$}; 
	\draw[-latex, semithick, dotted] (0,-6) -- (0,3) node[above]{$w$}; 
	\draw (0,0) node[above right]{O}; 
	\draw[domain=0:6, densely dashed] plot(\a,-\a) node[right]{$w = -a$};
	\draw[domain=0:-6, densely dashed] plot(\a,\a) node[left]{$w = a$};
	\draw[domain=0.001:2.68, variable=\theta] plot({-\theta*cos(\theta r)/sin(\theta r)},{-\theta/sin(\theta r)});
	\draw[domain=0.001:1.92, variable=\theta, thick] plot({-3*\theta*cos(\theta r)/sin(\theta r)},{-3*\theta/sin(\theta r)});
	\draw (-1,-1) node[left]{$(-1, -1)$};
	\fill (-1,-1) circle[radius=0.1];
	\draw (0,-pi/2) node[below left]{$-\frac{\pi}{2}$};
	\draw (-3,-3) node[left]{$(-3, -3)$};
	\fill (-3,-3) circle[radius=0.1];
	\draw (0,-3*pi/2) node[below left]{$-\frac{3\pi}{2}$};
\end{tikzpicture}
\caption{Stability boundary curves for $\tau = 1$ (solid) and $\tau = 1/3$ (bold)}
\label{fig: TE, real a and w}
\end{figure}
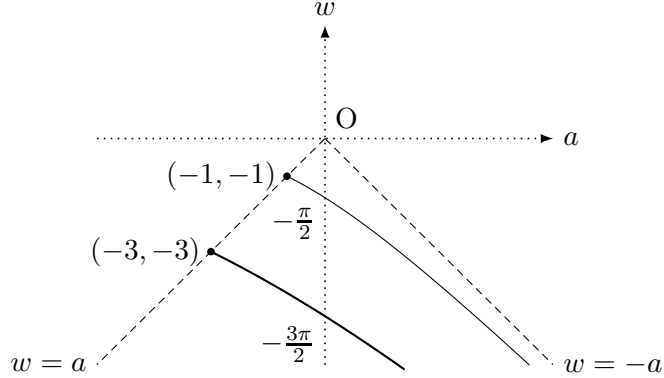

\subsubsection{Comparison with a study via Pontryagin's results}

For simplicity, let $\tau = 1$.
In \cite[Theorem A.5]{Hale--Lunel 1993}, the necessary and sufficient condition for which all the roots of Eq.~\eqref{eq: TE} have negative real parts is given as follows via Pontryagin's results (see \cite[Theorems A.3 and A.4]{Hale--Lunel 1993}): $a > -1$, $a - w > 0$, and
\begin{equation*}
	w > -C^{-1}(-a)\sin C^{-1}(-a) + a \cos C^{-1}(-a).
\end{equation*}
However, the parametrization~\eqref{eq: stability boundary curve for real a, w} is not directly obtained by this expression.
The above condition is same as that given in Corollary~\ref{cor: TE, real a and w} in view of
\begin{align*}
	& {-C^{-1}(-a)} \sin{C^{-1}(-a)} + a \cos{C^{-1}(-a)} \\
	&= -\frac{1}{\sin{C^{-1}(-a)}} \bigl( C^{-1}(-a)\sin^2{C^{-1}(-a)} - a\sin{C^{-1}(-a)}\cos{C^{-1}(-a)} \bigr) \\
	&= -R(-a),
\end{align*}
where $-a\sin{C^{-1}(-a)} = C^{-1}(-a)\cos{C^{-1}(-a)}$ is used.

The above procedure can be understood as a process eliminating the parameter $a$, which is explained as follows.
Suppose $a, w \in \R$.
By substituting $z = \iu\Omega$ ($\Omega \in \R \setminus \{0\}$) in Eq.~\eqref{eq: TE}, we have
\begin{align*}
	w
	&= (\iu\Omega + a)(\cos{\tau\Omega} + \iu\sin{\tau\Omega}) \\
	&= a\cos{\tau\Omega} - \Omega\sin{\tau\Omega} + \iu(a\sin{\tau\Omega} + \Omega\cos{\tau\Omega}).
\end{align*}
Here we are using the equivalent expression $(z + a)\e^{\tau z} - w = 0$ for Eq.~\eqref{eq: TE}.
Then the assumption $\Im(w) = 0$ leads to $a\sin{\tau\Omega} + \Omega\cos{\tau\Omega} = 0$, i.e.,
\begin{equation*}
	\Omega\cot{\tau\Omega} = -a
\end{equation*}
when $a \ne 0$.
Therefore, $w$ is expressed by
\begin{equation*}
	w
	= a\cos{\tau\Omega} - \Omega\sin{\tau\Omega}
	= - \frac{\Omega}{\sin{\tau\Omega}}
\end{equation*}
in the same way as above.
The above consideration is related to the \textit{method of D-partitions}.
See Section~\ref{sec: comparison with D-partitions} for the detail.

The above discussion is summarized in the following lemma.

\begin{lemma}
Suppose $\theta \not\in \pi\Z$ and $r \in \R$.
Then $\theta\cot{\theta} = r$ implies
\begin{equation*}
	\frac{\theta}{\sin\theta} = r\cos\theta + \theta\sin{\theta}.
\end{equation*}
Furthermore, if $\cos\theta \ne 0$, then the converse also holds.
\end{lemma}

\begin{proof}
We only need to show the converse under the assumption of $\cos\theta \ne 0$.
Then the equation is equivalent to
\begin{equation*}
	\frac{\theta}{\sin\theta}(1 - \sin^2\theta) = r\cos\theta.
\end{equation*}
By dividing the both sides by $\cos\theta$, we obtain $\theta\cot{\theta} = r$.
\end{proof}

\subsection{Remark on Hayes' result}

Hayes~\cite{Hayes 1950} considered a transcendental equation
\begin{equation*}
	s = c\e^s
\end{equation*}
for an unknown $s \in \C$ and a given constant $c \in \R$ to investigate Eq.~\eqref{eq: TE} with $\tau = 1$ for real $a$ and $w$.
Since this is equivalent to $(-s)\e^{-s} = -c$, the set of all roots of the above equation coincides with $-W(-c)$.
Hayes~\cite{Hayes 1950} did not use the concept of the Lambert $W$ function, but the study is considered to be the investigation of the principal complex branch $W_0(\zeta)$ for real $\zeta$.
Based on this approach, the following result has been obtained by Hayes~\cite[Lemma 2]{Hayes 1950}.

\begin{theorem}[\cite{Hayes 1950}]\label{thm: Re(W(zeta)) < sigma, real zeta}
Let $\zeta \in \R$ and $\sigma \in \R$ be given.
Then $\Re(z) < \sigma$ for all $z \in W(\zeta)$ if and only if
\begin{equation*}
	 \sigma > -1 \amd {-R(-\sigma)}\e^{\sigma} < \zeta < \sigma\e^{\sigma}
\end{equation*}
hold.
\end{theorem}

This is logically equivalent to Corollary~\ref{cor: TE, real a and w} by the discussions in Appendix~\ref{sec: Lambert W function}.

\section{Method by critical delay and stability region for real $a$ and imaginary $w$}\label{sec: method by critical delay for real a and imaginary w}

In this section, we study Eq.~\eqref{eq: TE} with real $a$ and complex $w$ and find the stability region via the method by critical delay.

We will use the following notation.

\begin{notation}
For each $\varphi \in (0, \pi)$, let $C(\argdot; \varphi) \colon [0, \varphi) \to \R$ be the function defined by
\begin{equation*}
	C(\theta; \varphi)
	\coloneqq \theta \cot(\theta - \varphi)
	= \theta \cdot \frac{\cos(\theta - \varphi)}{\sin(\theta - \varphi)}
\end{equation*}
for all $\theta \in [0, \varphi)$.
\end{notation}

\subsection{Property of the function \texorpdfstring{$C(\argdot; \varphi)$}{C(argdot; varphi)}}

We first study the case $\varphi \in (0, \pi/2]$.

\begin{lemma}\label{lem: C(argdot; varphi), 0 < varphi le pi/2}
Let $\varphi \in (0, \pi/2]$ be given.
Then the function $C(\argdot; \varphi) \colon [0, \varphi) \to \R$ is strictly monotonically decreasing and satisfies $\lim_{\theta \uparrow \varphi} C(\theta; \varphi) = -\infty$.
\end{lemma}

\begin{proof}
We have
\begin{align*}
	\der{\theta} \mspace{2mu} C(\theta; \varphi)
	&= \frac{\cos(\theta - \varphi)}{\sin(\theta - \varphi)} - \frac{\theta}{\sin^2(\theta - \varphi)} \\
	&= \frac{\sin(2\theta - 2\varphi) - 2\theta}{2\sin^2(\theta - \varphi)}
\end{align*}
for all $\theta \in [0, \varphi)$.
Let
\begin{equation*}
	f(\theta) \coloneqq \sin(2\theta - 2\varphi) - 2\theta
	\mspace{20mu}
	(\theta \in [0, \varphi)).
\end{equation*}
Then its derivative is 
\begin{equation*}
	f'(\theta)
	= 2(\cos(2\theta - 2\varphi) - 1)
	< 0.
\end{equation*}
Since $f(0) = -\sin(2\varphi) \le 0$, $f(\theta) < 0$ holds for all $\theta \in (0, \varphi)$.
Therefore, the monotonicity is obtained.
The limit is a consequence of $\lim_{\theta \uparrow 0} \cot\theta = -\infty$.
\end{proof}

Lemma~\ref{lem: C(argdot; varphi), 0 < varphi le pi/2} shows that for each given $\varphi \in (0, \pi/2]$, the function $C(\argdot; \varphi) \colon [0, \varphi) \to \R$ has its inverse function $C^{-1}(\argdot; \varphi) \colon (-\infty, 0] \to \R$ which is strictly monotonically decreasing and satisfies
\begin{equation}\label{eq: C^-1(argdot; varphi), 0 < varphi le pi/2}
	\lim_{r \to -\infty} C^{-1}(r; \varphi) = \varphi
	\amd
	C^{-1}(0; \varphi) = 0.
\end{equation}

To study the case $\varphi \in (\pi/2, \pi)$, we need to introduce the following notation.

\begin{definition}
For each $\varphi \in (\pi/2, \pi)$, let $\theta = S(\varphi)$ be the unique solution of
\begin{equation*}
	\sin(2\theta - 2\varphi) = 2\theta
\end{equation*}
in $(0, \varphi)$, or equivalently, $S(\varphi) \in (0, \varphi)$ satisfies
\begin{equation*}
	\sin(S(\varphi) - \varphi)\cos(S(\varphi) - \varphi) = S(\varphi).
\end{equation*}
\end{definition}

We note that the above $S(\varphi)$ coincides with $\phi^*$ used in \cite[Lemma 2]{Sakata 1998}.
The above unique existence of $S(\varphi)$ is ensured by the following lemma.

\begin{lemma}\label{lem: sin(2theta - 2varphi) - 2theta}
Let $\varphi \in (\pi/2, \pi)$ be given.
Then the following statements hold:
\begin{itemize}
\item $\sin(2\theta - 2\varphi) - 2\theta > 0$ for all $\theta \in [0, S(\varphi))$,
\item $\sin(2S(\varphi) - 2\varphi) = 2S(\varphi)$, and
\item $\sin(2\theta - 2\varphi) - 2\theta < 0$ for all $\theta \in (S(\varphi), \varphi)$.
\end{itemize}
\end{lemma}

\begin{proof}
Let $f(\theta) \coloneqq \sin(2\theta - 2\varphi) - 2\theta$ for all $\theta \in [0, \varphi)$.
Then
\begin{equation*}
	f'(\theta)
	= 2(\cos(2\theta - 2\varphi) - 1)
	< 0.
\end{equation*}
Since
\begin{equation*}
	f(0) = -\sin(2\varphi) > 0 \amd \lim_{\theta \uparrow \varphi} f(\theta) = -2\varphi < 0,
\end{equation*}
there exists a unique $\theta_\varphi \in (0, \varphi)$ such that $f(\theta_\varphi) = 0$ by the intermediate value theorem.
By the monotonicity, $\theta_\varphi = S(\varphi)$ holds.
Then $f(\theta) > 0$ for all $\theta \in [0, S(\varphi))$ and $f(\theta) < 0$ for all $\theta \in (S(\varphi), \varphi)$.
\end{proof}

By using the value $S(\varphi)$ for $\varphi \in (\pi/2, \pi)$, we obtain the following lemma.

\begin{lemma}\label{lem: C(argdot; varphi), pi/2 < varphi < pi}
Let $\varphi \in (\pi/2, \pi)$ be given.
Then the function $C(\argdot; \varphi) \colon [0, \varphi) \to \R$ is strictly monotonically increasing on $[0, S(\varphi)]$ and is strictly monotonically decreasing on $[S(\varphi), \varphi)$.
Furthermore, the function attains its maximum
\begin{equation*}
	M(\varphi) \coloneqq \cos^2(S(\varphi) - \varphi)
\end{equation*}
at $S(\varphi)$ and $\lim_{\theta \uparrow \varphi} C(\theta; \varphi) = -\infty$ holds.
\end{lemma}

\begin{proof}
By the proof of Lemma~\ref{lem: C(argdot; varphi), 0 < varphi le pi/2}, we have
\begin{equation*}
	\der{\theta} \mspace{2mu} C(\theta; \varphi)
	= \frac{\sin(2\theta - 2\varphi) - 2\theta}{2\sin^2(\theta - \varphi)}
\end{equation*}
for all $\theta \in [0, \varphi)$.
Therefore, the monotonicity properties stated in Lemma~\ref{lem: C(argdot; varphi), pi/2 < varphi < pi} follow by Lemma~\ref{lem: sin(2theta - 2varphi) - 2theta}.
The maximum can be calculated by using the relation $S(\varphi) = \sin(S(\varphi) - \varphi)\cos(S(\varphi) - \varphi)$.
The limit is a consequence of $\lim_{\theta \uparrow 0} \cot\theta = -\infty$.
\end{proof}

\begin{remark}
Since $C(\varphi - (\pi/2); \varphi) = 0$, we have
\begin{equation*}
	S(\varphi) < \varphi - \frac{\pi}{2}
\end{equation*}
for all $\varphi \in (\pi/2, \pi)$.
\end{remark}

See Fig.~\ref{fig: C(theta; varphi)} for the graphs of the function $C(\argdot; \varphi) \colon [0, \varphi) \to \R$ for each cases of $\varphi \in (0, \pi/2]$ and $\varphi \in (\pi/2, \pi)$.

\begin{figure}[tbp]
\centering
	\begin{subfigure}{0.4\columnwidth}
	\centering
		\begin{tikzpicture}
		\draw[-latex, semithick] (-1,0) -- (2.5,0) node [right]{$\theta$};
		\draw[-latex, semithick] (0,-3) -- (0,1) node [above]{$C(\theta; \varphi)$};
		\draw (0,0) node [below left]{O};
		\draw[densely dashed] (pi/2,-3) node [right]{$\theta = \varphi$} -- (pi/2,1);
		\draw[smooth, samples=100, variable=\theta, domain=0:1.19] plot (\theta, {\theta*cos((\theta - pi/2) r)/sin((\theta - pi/2) r)});
		\end{tikzpicture}
	\caption{$\varphi \in (0, \pi/2]$}
	\end{subfigure}
	\begin{subfigure}{0.4\columnwidth}
	\centering
		\begin{tikzpicture}
		\draw[-latex, semithick] (-1,0) -- (3,0) node[right]{$\theta$};
		\draw[-latex, semithick] (0,-3) -- (0,1) node[above]{$C(\theta; \varphi)$};
		\draw (0,0) node [below left]{O};
		\draw (pi/4,0) node [above right]{$\varphi - \frac{\pi}{2}$};
		\filldraw (pi/4,0) circle [radius=1.2pt];
		\draw[densely dashed] (3*pi/4,-3) node [right]{$\theta = \varphi$} -- (3*pi/4,1);
		\draw[smooth, samples=100, variable=\theta, domain=0:1.82] plot (\theta, {\theta*cos((\theta - 3*pi/4) r)/sin((\theta - 3*pi/4) r)});
		\end{tikzpicture}
	\caption{$\varphi \in (\pi/2, \pi)$. The function attains its maximum $M(\varphi) = \cos^2(S(\varphi) - \varphi)$ at $\theta = S(\varphi)$.}
	\end{subfigure}
\caption{Graphs of $C(\theta; \varphi) = \theta\cot(\theta - \varphi)$ for $\varphi \in (0, \pi/2]$ and $\varphi \in (\pi/2, \pi)$}
\label{fig: C(theta; varphi)}
\end{figure}
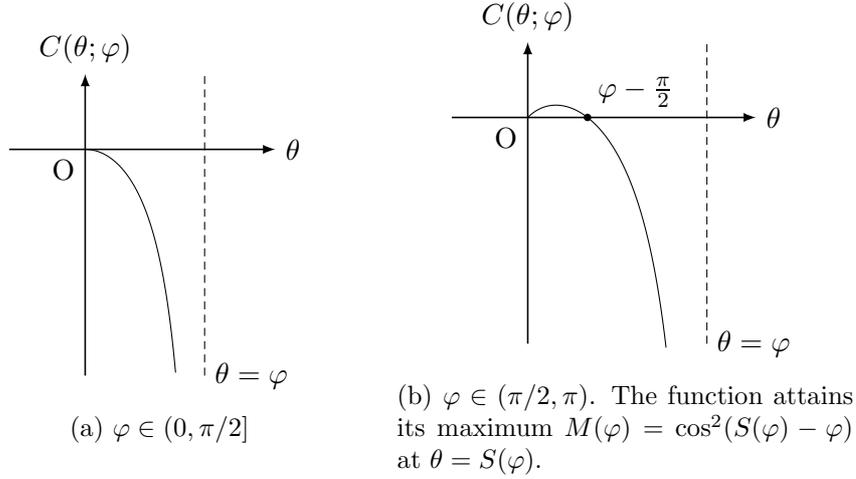

Based on Lemma~\ref{lem: C(argdot; varphi), pi/2 < varphi < pi}, we introduce the following.

\begin{definition}
For each given $\varphi \in (\pi/2, \pi)$, let $C_1(\argdot; \varphi)$ and $C_2(\argdot; \varphi)$ be the restrictions of the function $C(\argdot; \varphi) \colon [0, \varphi) \to \R$ to the intervals $[0, S(\varphi)]$ and $[S(\varphi), \varphi)$, respectively.
\end{definition}

Then Lemma~\ref{lem: C(argdot; varphi), pi/2 < varphi < pi} shows that the inverse functions $C_1^{-1}(\argdot; \varphi) \colon [0, M(\varphi)] \to \R$ and $C_2^{-1}(\argdot; \varphi) \colon (-\infty, M(\varphi)] \to \R$ have the following properties:
\begin{itemize}
\item $C_1^{-1}(\argdot; \varphi)$ is strictly monotonically increasing and satisfies
\begin{equation}\label{eq: C_1^-1(argdot; varphi)}
	C_1^{-1}(0; \varphi) = 0
	\amd
	C_1^{-1}(M(\varphi); \varphi) = S(\varphi).
\end{equation}
\item $C_2^{-1}(\argdot; \varphi)$ is strictly monotonically decreasing and satisfies
\begin{equation}\label{eq: C_2^-1(argdot; varphi)}
	\lim_{r \to -\infty} C_2^{-1}(r; \varphi) = \varphi
	\amd
	C_2^{-1}(M(\varphi); \varphi) = S(\varphi).
\end{equation}
\end{itemize}

\subsection{Inequality on critical delay and stability region for the case $\abs{\Arg(w)} \in (0, \pi/2]$}

We use the function introduced below.

\begin{definition}
For each $\varphi \in (0, \pi/2]$, let $R(\argdot; \varphi) \colon (-\infty, 0] \to \R$ be the function defined by
\begin{equation}\label{eq: R(r; varphi)}
	R(r; \varphi)
	\coloneqq -\frac{C^{-1}(r; \varphi)}{\sin(C^{-1}(r; \varphi) - \varphi)}
	= -\frac{r}{\cos(C^{-1}(r; \varphi) - \varphi)}.
\end{equation}
Here we have
\begin{equation*}
	-\frac{\pi}{2} \le -\varphi
	< C^{-1}(r; \varphi) - \varphi
	< 0
\end{equation*}
for all $r < 0$.
\end{definition}

The following lemma gives qualitative properties of the function $R(\argdot; \varphi) \colon (-\infty, 0] \to \R$ for each given $\varphi \in (0, \pi/2]$.

\begin{lemma}\label{lem: property of R(argdot; varphi)}
Let $\varphi \in (0, \pi/2]$ be given.
Then the function $R(\argdot; \varphi) \colon (-\infty, 0] \to \R$ is strictly monotonically decreasing and satisfies $\lim_{r \to -\infty} R(r; \varphi) = \infty$ and $R(0; \varphi) = 0$.
Furthermore,
\begin{equation*}
	\lim_{r \to -\infty} \frac{R(r; \varphi)}{|r|} = 1
	\amd
	\lim_{r \uparrow 0} \frac{R(r; \varphi)}{|r|} = \frac{1}{\cos\varphi}
\end{equation*}
hold.
Here we are interpreting that $1/{\cos{\varphi}} = \infty$ when $\varphi = \pi/2$.
\end{lemma}

\begin{proof}
We have
\begin{align*}
	\frac{\mathrm{d}}{\mathrm{d}\theta} \left( -\frac{\theta}{\sin(\theta - \varphi)} \right)
	&= -\frac{\sin(\theta - \varphi) - \theta\cos(\theta - \varphi)}{\sin^2(\theta - \varphi)} \\
	&= -\frac{1 - \theta\cot(\theta - \varphi)}{\sin(\theta - \varphi)} \\
	&> 0
\end{align*}
for all $\theta \in (0, \varphi)$ because $\theta\cot(\theta - \varphi) < 0$ and $\sin(\theta - \varphi) < 0$ hold.
Therefore, the monotonicity property of the function $R(\argdot; \varphi)$ follows by the monotonicity property of the function $C^{-1}(\argdot; \varphi)$.
The remaining properties are direct consequences of \eqref{eq: R(r; varphi)} and \eqref{eq: C^-1(argdot; varphi), 0 < varphi le pi/2}, where
\begin{equation*}
	\frac{R(r; \varphi)}{\abs{r}}
	= \frac{1}{\cos(C^{-1}(r; \varphi) - \varphi)}
\end{equation*}
holds for $r < 0$.
This completes the proof.
\end{proof}

\begin{theorem}\label{thm: tau_c(a, w) > 1, real a, Re(w) ge 0}
Let $a \in \R$ and $w \in \C \setminus \R$ be given so that $(a, w) \in D_\mathrm{c}$.
Suppose $\varphi \coloneqq \abs{\Arg(w)} \in (0, \pi/2]$.
Then $\tau_\mathrm{c}(a, w) > 1$ if and only if
\begin{equation*}
	a > 0
	\amd
	\abs{w} < R(-a; \varphi)
\end{equation*}
hold.
\end{theorem}

\begin{proof}
The inequality $\tau_\mathrm{c}(a, w) > 1$ becomes
\begin{equation*}
	\sqrt{\abs{w}^2 - a^2} - \varphi + \pi
	< \pi - \arccot{\left( \frac{a}{\sqrt{\abs{w}^2 - a^2}} \right)}
	= \arccot{\left( -\frac{a}{\sqrt{\abs{w}^2 - a^2}} \right)}
\end{equation*}
by the expression of $\tau_\mathrm{c}(a, w)$ and identity~\eqref{eq: arccot(x) + arccot(-x) = pi}.
Let $X(a, w) \coloneqq \sqrt{\abs{w}^2 - a^2}$.
Since $\cot|_{(0, \pi)} \colon (0, \pi) \to \R$ is strictly monotonically decreasing, the above inequality can be solved as
\begin{equation*}
	a > 0 \amd X(a, w) < C^{-1}(-a; \varphi).
\end{equation*}
By solving the last inequality with respect to $\abs{w}$, we obtain
\begin{align*}
	\abs{w}^2
	&< a^2 + C^{-1}(-a; \varphi)^2 \\
	&= C^{-1}(-a; \varphi)^2 \bigl[ \cot^2(C^{-1}(-a; \varphi) - \varphi) + 1 \bigr],
\end{align*}
which is equivalent to $\abs{w} < R(-a; \varphi)$.
\end{proof}

\begin{corollary}[\cite{Sakata 1998}]\label{cor: TE, real a, Re(w) ge 0}
Suppose $a \in \R$, $w \in \C \setminus \R$, and $\varphi \coloneqq \abs{\Arg(w)} \in (0, \pi/2]$.
Then all the roots of Eq.~\eqref{eq: TE} have negative real parts if and only if 
\begin{equation*}
	a > 0
	\amd
	\abs{w} < \frac{1}{\tau}R(-\tau a; \varphi)
\end{equation*}
hold.
\end{corollary}

\begin{proof}
From Theorem~\ref{thm: T(a, w), real a}, all the roots of Eq.~\eqref{eq: TE} have negative real parts if and only if one of the following conditions is satisfied:
\begin{itemize}
\item $a \ge \abs{w}$ and $a > \Re(w)$.
\item $(a, w) \in D_\mathrm{c}$ and $\tau_\mathrm{c}(a, w) > \tau$.
\end{itemize}
Here $a > \Re(w)$ is automatically satisfied under the condition $a \ge \abs{w}$, and the second condition becomes
\begin{equation*}
	\tau\Re(w) < \tau a < \tau{\abs{w}}, \mspace{10mu} \tau a > 0,
	\mspace{10mu} \text{and} \mspace{10mu}
	\tau{\abs{w}} < R(-\tau a; \varphi)
\end{equation*}
from equality~\eqref{eq: tau_c(a tau, w tau)} and Theorem~\ref{thm: tau_c(a, w) > 1, real a, Re(w) ge 0}.
By combining this and the first condition, we obtain the conclusion from Lemma~\ref{lem: property of R(argdot; varphi)} (see also Fig.~\ref{fig: decomposition of (a, |w|)-plane by critical delay}).
\end{proof}

Since the function $C^{-1}(\argdot; \varphi)$ gives a one-to-one correspondence between the open intervals $(-\infty, 0)$ and $(0, \varphi)$, the curve
\begin{equation*}
	\Set{\left( a, \frac{1}{\tau}R(-\tau a; \varphi) \right)}{a > 0}
\end{equation*}
in $(a, \abs{w})$-plane is parametrized by
\begin{equation}\label{eq: stability boundary curve for real a and imaginary w}
	a = -\frac{1}{\tau}\theta\cot(\theta - \varphi) \amd \abs{w} = -\frac{\theta}{\tau\sin(\theta - \varphi)}
\end{equation}
for $\theta \in (0, \varphi)$ in view of
\begin{equation*}
	-\tau a
	= C^{-1}(-\tau a; \varphi)\cot {\left( C^{-1}(-\tau a; \varphi) - \varphi \right)}.
\end{equation*}
See Fig.~\ref{fig: boundary curve, real a, Re(w) ge 0} for the picture of boundary curves when $\varphi = \pi/2$ and $\varphi = \pi/4$.

\begin{figure}[tbp]
\centering
	\begin{tikzpicture}
	\draw[-latex, semithick] (-1,0) -- (3,0) node[right]{$a$}; 
	\draw[-latex, semithick] (0,-1) -- (0,3) node[above]{$\abs{w}$}; 
	\draw (0,0) node[below right]{O}; 
	\draw[smooth,samples=100,variable=\a,domain=0:3,dotted,semithick] plot(\a,\a) node[right]{$\abs{w} = a$};
	\draw[smooth, samples=100, variable=\theta, domain=0.001:1.182] plot({-\theta/tan((\theta - pi/2) r)},{(-\theta/sin((\theta - pi/2) r)))});
	\draw[smooth, samples=100, variable=\theta, domain=0.001:0.6, densely dashed] plot({-\theta/tan((\theta - pi/4) r)},{(-\theta/sin((\theta - pi/4) r))});
	\end{tikzpicture}
\caption{Boundary curves in Corollary~\ref{cor: TE, real a, Re(w) ge 0} when $\varphi = \pi/2$ (solid) and $\varphi = \pi/4$ (dashed) for the case that $\tau = 1$. The dotted line denotes the straight line of $\abs{w} = a$.}
\label{fig: boundary curve, real a, Re(w) ge 0}
\end{figure}

\subsection{Inequality on critical delay and stability region for the case $\abs{\Arg(w)} \in (\pi/2, \pi)$}

We use the functions introduced below.

\begin{definition}
For each $\varphi \in (\pi/2, \pi)$, let $R_1(\argdot; \varphi) \colon [0, M(\varphi)] \to \R$ be the function defined by
\begin{equation}\label{eq: R_1(r; varphi)}
	R_1(r; \varphi)
	\coloneqq -\frac{C_1^{-1}(r; \varphi)}{\sin{\left( C_1^{-1}(r; \varphi) - \varphi \right)}}
	= -\frac{r}{\cos{\left( C_1^{-1}(r; \varphi) - \varphi \right)}}
\end{equation}
Here we have
\begin{equation*}
	-\pi < -\varphi < C_1^{-1}(r; \varphi) - \varphi < S(\varphi) - \varphi < -\frac{\pi}{2}
\end{equation*}
for all $0 < r < M(\varphi)$.
\end{definition}

\begin{definition}
For each $\varphi \in (\pi/2, \pi)$, let $R_2(\argdot; \varphi) \colon (-\infty, M(\varphi)] \to \R$ be the function defined by
\begin{equation}\label{eq: R_2(r; varphi)}
	R_2(r; \varphi)
	\coloneqq -\frac{C_2^{-1}(r; \varphi)}{\sin{\left( C_2^{-1}(r; \varphi) - \varphi \right)}}
	= -\frac{r}{\cos{\left( C_2^{-1}(r; \varphi) - \varphi \right)}}.
\end{equation}
Here we have
\begin{equation*}
	-\pi < S(\varphi) - \varphi < C_2^{-1}(r; \varphi) - \varphi < -\frac{\pi}{2}
\end{equation*}
for all $0 < r < M(\varphi)$ and
\begin{equation*}
	-\frac{\pi}{2} \le C_2^{-1}(r; \varphi) - \varphi < 0
\end{equation*}
for all $r \le 0$.
\end{definition}

The following lemmas give qualitative properties of the functions $R_1(\argdot; \varphi)$ and $R_2(\argdot; \varphi)$.
The proofs are similar to that of Lemma~\ref{lem: property of R(argdot; varphi)} but with \eqref{eq: C_1^-1(argdot; varphi)} and \eqref{eq: C_2^-1(argdot; varphi)}.

\begin{lemma}\label{lem: property of R_1(argdot; varphi)}
Let $\varphi \in (\pi/2, \pi)$ be given.
Then the function $R_1(\argdot; \varphi) \colon [0, M(\varphi)] \to \R$ is strictly monotonically increasing and satisfies
\begin{equation*}
	R_1(0; \varphi) = 0
	\amd
	R_1(M(\varphi); \varphi) = \sqrt{M(\varphi)}.
\end{equation*}
Furthermore,
\begin{equation*}
	\lim_{r \downarrow 0} \frac{R_1(r; \varphi)}{r} = \frac{1}{\abs{\cos\varphi}}
	\amd
	R_1(r; \varphi) > \frac{r}{\abs{\cos\varphi}}
\end{equation*}
hold.
\end{lemma}

\begin{proof}
The monotonicity property follows by the similar way to the proof of Lemma~\ref{lem: property of R(argdot; varphi)}.
All the remaining properties are consequences of \eqref{eq: C_1^-1(argdot; varphi)} and  \eqref{eq: R_1(r; varphi)} in view of $\sqrt{M(\varphi)} = -\cos(S(\varphi) - \varphi)$.
\end{proof}

The proof of the following lemma is similar to that of Lemma~\ref{lem: property of R_1(argdot; varphi)}.
Therefore, it can be omitted.

\begin{lemma}\label{lem: property of R_2(argdot, varphi)}
Let $\varphi \in (\pi/2, \pi)$ be given.
Then the function $R_2(\argdot; \varphi) \colon (-\infty, M(\varphi)] \to \R$ is strictly monotonically decreasing and satisfies
\begin{equation*}
	\lim_{r \to -\infty} R_2(r; \varphi) = \infty
	\amd
	R_2(M(\varphi); \varphi) = \sqrt{M(\varphi)}.
\end{equation*}
Furthermore,
\begin{equation*}
	\lim_{r \to -\infty} \frac{R_2(r; \varphi)}{\abs{r}} = 1
\end{equation*}
holds.
\end{lemma}

\begin{theorem}\label{thm: tau_c(a, w) > 1, real a, Re(w) < 0}
Let $a \in \R$ and $w \in \C \setminus \R$ be given so that $(a, w) \in D_\mathrm{c}$.
Suppose $\varphi \coloneqq \abs{\Arg(w)} \in (\pi/2, \pi)$.
Then $\tau_\mathrm{c}(a, w) > 1$ if and only if one of the following conditions is satisfied:
\begin{enumerate}[label=\textup{(\roman*)}]
\item $a \ge 0$ and $\abs{w} < R_2(-a; \varphi)$.
\item $-M(\varphi) < a < 0$ and $R_1(-a; \varphi) < \abs{w} < R_2(-a; \varphi)$.
\end{enumerate}
\end{theorem}

\begin{proof}
In the same way as the proof of Theorem~\ref{thm: tau_c(a, w) > 1, real a, Re(w) ge 0}, we obtain the inequality
\begin{equation*}
	X(a, w) - \varphi + \pi < \arccot{\left( -\frac{a}{X(a, w)} \right)},
\end{equation*}
where $X(a, w) \coloneqq \sqrt{\abs{w}^2 - a^2}$.
From Lemma~\ref{lem: C(argdot; varphi), pi/2 < varphi < pi}, this can be solved as
\begin{enumerate}
\item[(i)] $a \ge 0$ and $X(a, w) < C_2^{-1}(-a; \varphi)$, or
\item[(ii)] $-M(\varphi) < a < 0$ and $C_1^{-1}(-a; \varphi) < X(a, w) < C_2^{-1}(-a; \varphi)$.
\end{enumerate}
\begin{itemize}
\item Case (i): By solving $X(a, w) < C_2^{-1}(-a; \varphi)$ with respect to $\abs{w}$, we obtain
\begin{equation*}
	\abs{w} < R_2(-a; \varphi)
\end{equation*}
in the similar way to the proof of Theorem~\ref{thm: tau_c(a, w) > 1, real a, Re(w) ge 0}.
\item Case (ii): By solving $C_1^{-1}(-a; \varphi) < X(a, w) < C_2^{-1}(-a; \varphi)$ with respect to $\abs{w}$, we obtain
\begin{equation*}
	R_1(-a; \varphi) < \abs{w} < R_2(-a; \varphi)
\end{equation*}
in the similar way to to the proof of Theorem~\ref{thm: tau_c(a, w) > 1, real a, Re(w) ge 0}.
\end{itemize}
This completes the proof.
\end{proof}

The following is a consequence of Theorem~\ref{thm: tau_c(a, w) > 1, real a, Re(w) < 0}.
It is proved in the similar way to Corollary~\ref{cor: TE, real a, Re(w) ge 0}, and therefore, the proof can be omitted.

\begin{corollary}[\cite{Sakata 1998}]\label{cor: TE, real a, Re(w) < 0}
Suppose $a \in \R$, $w \in \C \setminus \R$, and $\varphi \coloneqq \abs{\Arg(w)} \in (\pi/2, \pi)$.
Then all the roots of Eq.~\eqref{eq: TE} have negative real parts if and only if one of the following conditions is satisfied:
\begin{enumerate}[label=\textup{(\roman*)}]
\item
\begin{equation*}
	a \ge 0
	\amd
	\abs{w} < \frac{1}{\tau}R_2(-\tau a; \varphi).
\end{equation*}
\item
\begin{equation*}
	-\frac{1}{\tau}M(\varphi) < a < 0
	\amd
	\frac{1}{\tau}R_1(-\tau a; \varphi) < \abs{w} < \frac{1}{\tau}R_2(-\tau a; \varphi).
\end{equation*}
\end{enumerate}
\end{corollary}

Finally, we discuss the parametrization of the boundary curve.
Since
\begin{itemize}
\item the function $C_2^{-1}(\argdot; \varphi)$ gives a one-to-one correspondence between $(-\infty, M(\varphi))$ and $(S(\varphi), \varphi)$,
\item the function $C_1^{-1}(\argdot; \varphi)$ gives a one-to-one correspondence between $(0, M(\varphi))$ and $(0, S(\varphi))$,
\end{itemize}
the curves
\begin{align*}
	&\Set{\left( a, \frac{1}{\tau}R_2(-\tau a; \varphi) \right)}{a > -\frac{1}{\tau}M(\varphi)}, \\
	&\Set{\left( a, \frac{1}{\tau}R_1(-\tau a; \varphi) \right)}{-\frac{1}{\tau}M(\varphi) < a < 0}
\end{align*}
in $(a, \abs{w})$-plane are parametrized by \eqref{eq: stability boundary curve for real a and imaginary w} for $\theta \in (S(\varphi), \varphi)$ and for $\theta \in (0, S(\varphi))$, respectively.
By taking $\theta \in (0, \varphi)$, we also obtain the parametrization of the joined curve.
See Fig.~\ref{fig: boundary curve, real a, Re(w) < 0} for the picture of boundary curves when $\varphi = 9\pi/10$ and $\varphi = 3\pi/4$.

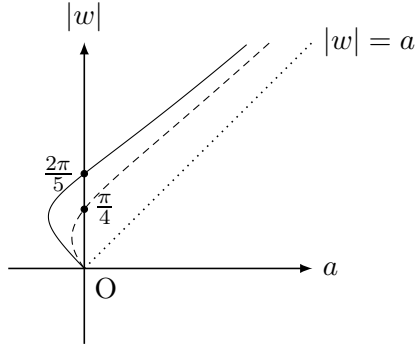
\begin{figure}[tbp]
\centering
	\begin{tikzpicture}
	\draw[-latex, semithick] (-1,0) -- (3,0) node[right]{$a$}; 
	\draw[-latex, semithick] (0,-1) -- (0,3) node[above]{$\abs{w}$}; 
	\draw (0,0) node[below right]{O}; 
	\draw (0,pi/4) node[right]{$\frac{\pi}{4}$};
	\filldraw (0,pi/4) circle [radius=1.2pt];
	\draw (0,2*pi/5) node[left]{$\frac{2\pi}{5}$};
	\filldraw (0,2*pi/5) circle [radius=1.2pt];
	\draw[variable=\a,domain=0:3,dotted,semithick] plot(\a,\a) node[right]{$\abs{w} = a$};
	\draw[smooth, samples=100, variable=\theta, domain=0.001:2.06] plot({-\theta/tan((\theta - 9*pi/10) r)},{-(\theta/sin((\theta - 9*pi/10) r))});
	\draw[smooth, samples=100, variable=\theta, domain=0.001:1.74, densely dashed] plot({-\theta/tan((\theta - 3*pi/4) r)},{-(\theta/sin((\theta - 3*pi/4) r))});
	\end{tikzpicture}
\caption{Boundary curves in Corollary~\ref{cor: TE, real a, Re(w) < 0} when $\varphi = 9\pi/10$ (solid) and $\varphi = 3\pi/4$ (dashed) for the case that $\tau = 1$. The dotted line denotes the straight line of $\abs{w} = a$.}
\label{fig: boundary curve, real a, Re(w) < 0}
\end{figure}

\subsection{Remarks}

\paragraph{Expression of critical delay}

By following the arguments of the proofs of Theorems~\ref{thm: tau_c(a, w) > 1, real a, Re(w) ge 0} and \ref{thm: tau_c(a, w) > 1, real a, Re(w) < 0} in reverse, one can obtain the expression of the critical delay from Corollaries~\ref{cor: TE, real a, Re(w) ge 0} and \ref{cor: TE, real a, Re(w) < 0} in the same reasoning as Remark~\ref{rmk: critical delay, real a and w}.

\paragraph{On Sakata's result}

Corollaries~\ref{cor: TE, real a, Re(w) ge 0} and \ref{cor: TE, real a, Re(w) < 0} are due to Sakata~\cite[Theorem]{Sakata 1998}.
The visualization of boundary curves was also obtained by Sakata~\cite{Sakata 1998} without the parametrization and the parameter range.

\paragraph{On parameter range given by Matsunaga}

Matsunaga \cite[Theorem B]{Matsunaga 2008} gave a restatement of Sakata's result with the following parametrization of the boundary curve:
\begin{equation*}
	a = -\frac{1}{\tau} \theta\cot(\theta - \abs{\psi})
	\amd
	b = \frac{\theta}{\tau \sin(\theta - \abs{\psi})}
\end{equation*}
for $\theta \in (\abs{\psi} - \pi, \abs{\psi})$.
Here $w = -b\mathrm{e}^{\iu\psi}$ for $b \in \R \setminus \{0\}$ and $\psi \in [-\pi/2, \pi/2]$.
We now compare this with the parametrization~\eqref{eq: stability boundary curve for real a and imaginary w} with the parameter range $\theta \in (0, \varphi)$.
\begin{itemize}
\item Case 1: $b > 0$.
By the proof of Lemma~\ref{lem: critical delay in Matsunaga}, we have
\begin{equation*}
	b = \abs{w} \amd -{\abs{\psi}} = \varphi - \pi.
\end{equation*}
Therefore, the above parametrization becomes
\begin{equation*}
	a = -\frac{1}{\tau} \theta\cot(\theta + \varphi)
	\amd
	\abs{w} = -\frac{\theta}{\tau \sin(\theta + \varphi)}
\end{equation*}
for $\theta \in (-\varphi, -\varphi + \pi)$.
\item Case 2: $b < 0$.
By the proof of Lemma~\ref{lem: critical delay in Matsunaga}, we have
\begin{equation*}
	b = -{\abs{w}} \amd |\psi| = \varphi.
\end{equation*}
Therefore, the above parametrization becomes
\begin{equation*}
	a = -\frac{1}{\tau} \theta\cot(\theta - \varphi)
	\amd
	\abs{w} = -\frac{\theta}{\tau \sin(\theta - \varphi)}
\end{equation*}
for $\theta \in (\varphi - \pi, \varphi)$.
\end{itemize}
We note that the parametrization in Case 1 is also expressed by that in Case 2 because the appearing functions are even.
However, the parameter range $(\varphi - \pi, \varphi)$ is not consistent with $(0, \varphi)$.

\section{Comparison with the method of D-partitions}\label{sec: comparison with D-partitions}

In this section, we apply the method of D-partitions to Eq.~\eqref{eq: TE} for the case of real $a$ and complex $w$ and to compare this with the results obtained in Sections~\ref{sec: method by critical delay for real a and w} and \ref{sec: method by critical delay for real a and imaginary w}.

\subsection{Method of D-partitions}\label{subsec: D-partitions}

In this subsection, we briefly summarize the method of D-partitions.
We consider a transcendental equation having $n$-th real parameters $(p_1, \dots, p_n)$.
As is mentioned in Introduction, the delay parameters are not included in $(p_1, \dots, p_n)$ for the purpose of obtaining the stability region.
In other words, the delay parameters are fixed in the following consideration of the method of D-partitions.

Assuming the situation that the transcendental equation has a root $\iu\Omega$ on the imaginary axis, we have the two constraints
\begin{equation*}
	f_1(p_1, \dots, p_n, \Omega) = 0
	\mspace{10mu} \text{and} \mspace{10mu}
	f_2(p_1, \dots, p_n, \Omega) = 0
\end{equation*}
which are obtained by the real and imaginary parts of the left-hand side of the considering transcendental equation.
Here it is assumed that the domain of the function $f = (f_1, f_2)$ is open in the extended parameter space (i.e., $(p_1, \dots, p_n, \Omega)$-space), and $f$ is sufficiently smooth in its domain.
Then the regular level set theorem states that the set of all solutions $(p_1, \dots, p_n, \Omega)$ satisfying the above constraints is an $(n - 1)$-dimensional smooth submanifold embedded in the extended parameter space if $0$ is a regular value of the function $f$.
Furthermore, if one can choose indices $1 \le i < j \le n$ so that the Jacobian determinant
\begin{equation*}
	\left| \frac{\partial(f_1, f_2)}{\partial(p_i, p_j)} \right|
\end{equation*}
is nonzero at some point, then by applying the implicit function theorem, the solution set are locally represented by the graph of some functions
\begin{equation*}
	p_i = p_i((p_k)_{k \ne i, j}, \Omega)
	\mspace{10mu} \text{and} \mspace{10mu}
	p_j = p_j((p_k)_{k \ne i, j}, \Omega).
\end{equation*}
If this can be possible globally, then one obtain hyper-surfaces in the parameter $(p_1, \dots, p_n)$-space by removing the angular frequency $\Omega$ from the extended parameter space.

We refer the reader to \cite[Chapter III.3]{Elsgolts--Norkin 1973}, \cite[Subsection 3.2 in Chapter 2]{Kolmanovskii--Nosov 1986}, and \cite[Sections XI.1 and XI.2 in Chapter XI]{Diekmann--vanGils--Lunel--Walther 1995} for the details of the method of D-partitions including the analysis of Eq.~\eqref{eq: TE} for the case that $a$ and $w$ are real numbers.
See \cite{Diekmann--Korvasova 2013} for a note stressing the importance of converting the parameters $a$ and $w$ to original model parameters in the method of $D$-partitions.
See also \cite{Bernard--Belair--Mackey 2001}, \cite{Kiss--Krauskopf 2010} for applications of the method of D-partitions to differential equations with distributed delay and \cite{Diekmann--Getto--Nakata 2016} to a neutral delay differential equation.

\subsection{Conditions for roots on the imaginary axis with real \texorpdfstring{$a$}{a}}\label{subsec: roots on imaginary axis}

Suppose $a \in \R$ and $w \in \C \setminus \{0\}$.
The proof of Theorem~\ref{thm: T(a, w), real a} in \cite{Matsunaga 2008} relies on the investigation of the condition on $a$, $w$, and $\tau$ for which Eq.~\eqref{eq: TE} has a purely imaginary root (i.e., a nonzero root on the imaginary axis).
We note that Eq.~\eqref{eq: TE} has a root $0$ if and only if $w = a$.

\subsubsection{An interpretation via implicit function theorem}

Angular frequency eqs.~\eqref{eq: angular frequency}
\begin{equation*}
	\left\{
	\begin{aligned}
		a - \abs{w}\cos(\Arg(w) - \tau\Omega) &= 0, \\
		\Omega - \abs{w}\sin(\Arg(w) - \tau\Omega) &= 0
	\end{aligned}
	\right.
\end{equation*}
can be considered as a system of equations with respect to the five variables $a \in \R$, $\rho \coloneqq \abs{w} > 0$, $\psi \coloneqq \Arg(w) \in (-\pi, \pi]$, $\tau > 0$, and $\Omega \in \R$.
We now give an interpretation on Lemmas~\ref{lem: angular frequency, [0, pi]} and \ref{lem: angular frequency, (-pi, 0)} from the viewpoint of the implicit function theorem.

Let
\begin{align*}
	f(a, \rho, \psi, \tau, \Omega)
	&\coloneqq (f_1(a, \rho, \psi, \tau, \Omega), f_2(a, \rho, \psi, \tau, \Omega)) \\
	&\coloneqq (a - \rho\cos(\psi - \tau\Omega), \Omega - \rho\sin(\psi - \tau\Omega)).
\end{align*}
We study the solution set of the equation $f(a, \rho, \psi, \tau, \Omega) = 0$.
The Jacobian determinant $|\partial(f_1, f_2)/\partial(\tau, \Omega)|$ is calculated as
\begin{align*}
	&\left| \frac{\partial(f_1, f_2)}{\partial(\tau, \Omega)} \right| \\
	&= -\rho\Omega\sin(\psi - \tau\Omega)[1 + \rho\tau\cos(\psi - \tau\Omega)] - \{-\rho\tau\sin(\psi - \tau\Omega) \cdot \rho\Omega\cos(\psi - \tau\Omega) \} \\
	&= -\rho\Omega\sin(\psi - \tau\Omega).
\end{align*}
Therefore, by restricting the domain of definition of the function $f$ to the subset satisfying $\Omega\sin(\psi - \tau\Omega) \ne 0$, both of $\tau$ and $\Omega$ can be written as functions of $(a, \rho, \psi)$.

The independency of $\Omega$ from $\psi$ is also derived by calculating the partial derivative $\partial\Omega/\partial\psi$ as follows:
By partially differentiating $f(a, \rho, \psi, \tau, \Omega) = 0$ with respect to $\psi$, we have
\begin{align*}
	&\rho\sin(\psi - \tau\Omega) \cdot \left[ 1 - \frac{\partial(\tau\Omega)}{\partial\psi} \right] = 0, \\
	&\frac{\partial\Omega}{\partial\psi} - \rho\cos(\psi - \tau\Omega) \cdot \left[ 1 - \frac{\partial(\tau\Omega)}{\partial\psi} \right] = 0.
\end{align*}
Therefore, we necessarily have $\partial\Omega/\partial\psi = 0$ if $\sin(\psi - \tau\Omega) \ne 0$.

\subsection{Curves parametrized by angular frequency}\label{subsec: curves parametrized by angular frequency}

Suppose $a \in \R$ and $w \in \C \setminus \{0\}$.
For each given $\Omega \in \R$ and $\tau > 0$, we will find a condition on $a$ and $w$ for which Eq.~\eqref{eq: TE} has a root $\iu\Omega$.
Since Eq.~\eqref{eq: TE} has a root $0$ if and only if $a = w$, it is sufficient to find a purely imaginary root of Eq.~\eqref{eq: TE}.

We use the following notation.

\begin{notation}
Let $\tau > 0$ be given.
For each $\Omega \in \R \setminus \{0\}$ and $\psi \in (-\pi, \pi]$ satisfying $\tau\Omega - \psi \not\in \pi\Z$, let
\begin{equation*}
	a(\Omega, \psi; \tau) \coloneqq -\Omega\cot(\tau\Omega - \psi)
	\amd
	\rho(\Omega, \psi; \tau) \coloneqq - \frac{\Omega}{\sin(\tau\Omega - \psi)}.
\end{equation*}
For each $\tau > 0$, $a(\Omega, \psi; \tau)$ and $\rho(\Omega, \psi; \tau)$ are expressed as
\begin{equation}\label{eq: a(theta, psi; tau), rho(theta, psi; tau)}
	a(\Omega, \psi; \tau) = \frac{1}{\tau}a(\tau\Omega, \psi; 1)
	\amd
	\rho(\Omega, \psi; \tau) = \frac{1}{\tau}\rho(\tau\Omega, \psi; 1).
\end{equation}
\end{notation}

From \eqref{eq: angular frequency}, Eq.~\eqref{eq: TE} has a root $\iu\Omega$ ($\Omega \ne 0$) if and only if $\sin(\Arg(w) - \tau\Omega) \ne 0$, i.e., $\tau\Omega - \Arg(w) \not\in \pi\Z$,
\begin{gather*}
	\abs{w}
	= \frac{\Omega}{\sin(\Arg(w) - \tau\Omega)}
	= \rho(\Omega, \Arg(w); \tau), \\
\intertext{and}
	a
	= \abs{w}\cos(\Arg(w) - \tau\Omega)
	= a(\Omega, \Arg(w); \tau).
\end{gather*}
This means that by varying $\Omega \in \R \setminus \{0\}$ so that $\tau\Omega - \Arg(w) \not\in \pi\Z$ for each given $\tau > 0$ and each fixed $\Arg(w)$, we obtain the parametrization of curves in $(a, \abs{w})$-plane on which Eq.~\eqref{eq: TE} has purely imaginary roots.
This is the method of D-partitions in our case.

We introduce the following notation.

\begin{notation}
For each integer $k \ne 0$ and $\psi \in (-\pi, \pi]$, let
\begin{equation*}
	I_k(\psi) \coloneqq
	\begin{cases}
		(-\pi, 0) + \psi + 2k\pi & (k \ge 1), \\
		(0, \pi) + \psi + 2k\pi & (k \le -1).
	\end{cases}
\end{equation*}
Let
\begin{equation*}
	I_\mathrm{c}(\psi) \coloneqq
	\begin{cases}
		(0, \psi) & (\psi \ge 0), \\
		(\psi, 0) & (\psi \le 0).
	\end{cases}
\end{equation*}
Here we are interpreting that $I_\mathrm{c}(\psi)$ is empty when $\psi = 0$.
\end{notation}

\begin{lemma}\label{lem: positivity of rho(theta, psi; tau)}
Let $\psi \in (-\pi, \pi]$ and $\tau > 0$ be given.
Suppose $\Omega \in \R \setminus \{0\}$.
Then $\rho(\Omega, \psi; \tau) > 0$ if and only if
\begin{equation*}
	\tau\Omega \in I_\mathrm{c}(\psi) \cup \bigcup_{k \in \Z \setminus \{0\}} I_k(\psi)
\end{equation*}
holds.
\end{lemma}

\begin{proof}
From \eqref{eq: a(theta, psi; tau), rho(theta, psi; tau)}, it is sufficient to consider the case $\tau = 1$.
We also consider the case $\psi > 0$.
The proof is divided into the following two cases.
\begin{itemize}[leftmargin=*]
\item Case 1: $\Omega > 0$.
In this case, the positivity of $\rho(\Omega, \psi; 1)$ is equivalent to $\sin(\Omega - \psi) < 0$.
This is equivalent to
\begin{equation*}
	\Omega - \psi \in (-\pi, 0) + 2k\pi
\end{equation*}
for some $k \in \Z$.
Since $\Omega - \psi > -\psi \ge -\pi$, $k$ necessarily satisfies $k \ge 0$.
We note that the condition for the case $k = 0$ becomes $\Omega \in (0, \psi) = I_\mathrm{c}(\psi)$ because of $\psi - \pi \le 0$.
\item Case 2: $\Omega < 0$.
The same reasoning imposes $\sin(\Omega - \psi) > 0$, i.e.,
\begin{equation*}
	\Omega - \psi \in (0, \pi) + 2k\pi
\end{equation*}
for some $k \in \Z$.
Since $\Omega - \psi < -\psi < 0$, $k$ necessarily satisfies $k \le -1$.
\end{itemize}
The similar proof is valid for the case $\psi < 0$, and the proof for the case $\psi = 0$ is more simpler than these cases.
This completes the proof.
\end{proof}

In view of Lemma~\ref{lem: positivity of rho(theta, psi; tau)}, we introduce the following notation.

\begin{notation}
For each $\psi \in (-\pi, \pi]$ and each $\tau > 0$, let
\begin{align*}
	\Gamma_*(\psi; \tau)
	&\coloneqq \Set{\bigl( a(\Omega, \psi; \tau), \rho(\Omega, \psi; \tau)\e^{\iu\psi} \bigr)}{\tau\Omega \in I_*(\psi)}, \\
	\widetilde{\Gamma}_*(\psi; \tau)
	&\coloneqq \Set{\bigl( a(\Omega, \psi; \tau), \rho(\Omega, \psi; \tau) \bigr)}{\tau\Omega \in I_*(\psi)},
\end{align*}
where the symbol $*$ denotes $\mathrm{c}$ or some nonzero integer $k$.
\end{notation}

For each $* \in \{\mathrm{c}\} \cup (\Z \setminus \{0\})$, we consider $\Gamma_*(\psi; \tau)$ as a parametrized curve with the parametrization given by
\begin{equation}\label{eq: parametrization of Gamma_*(psi; tau)}
	\frac{1}{\tau}I_*(\psi) \ni \Omega \mapsto \bigl( a(\Omega, \psi; \tau), \rho(\Omega, \psi; \tau)\e^{\iu\psi} \bigr) \in \R \times (\C \setminus \{0\}).
\end{equation}
We note that
\begin{equation}\label{eq: Gamma_*(psi ;tau)}
	\Gamma_*(\psi; \tau) = \frac{1}{\tau}\Gamma_*(\psi; 1)
\end{equation}
holds from \eqref{eq: a(theta, psi; tau), rho(theta, psi; tau)}.

By using the above notation, under the assumption of $\Arg(w) = \psi$, the set of all $(a, w)$ for which Eq.~\eqref{eq: TE} has purely imaginary roots is represented by
\begin{equation*}
	\Gamma_\mathrm{c}(\psi; \tau) \cup \bigcup_{k \in \Z \setminus \{0\}} \Gamma_k(\psi; \tau).
\end{equation*}

Here we briefly study the location of each curves.
For all $\Omega$ satisfying $\tau\Omega \in I_\mathrm{c}(\psi) \cup \bigcup_{k \in \Z \setminus \{0\}} I_k(\psi)$, the inequalities
\begin{equation*}
	-\rho(\Omega, \psi; \tau) < a(\Omega, \psi; \tau) < \rho(\Omega, \psi; \tau)
\end{equation*}
are obtained.
This can be seen by dividing all the terms by $\Omega/\sin(\tau\Omega - \psi) < 0$ because the resulting inequalities are
\begin{equation*}
	-1 < -\cos(\tau\Omega - \psi) < 1.
\end{equation*}
The above inequalities are also obtained by
\begin{equation}\label{eq: rho(Omega, psi; tau)^2 - a(Omega, psi; tau)^2}
	\rho(\Omega, \psi; \tau)^2 - a(\Omega, \psi; \tau)^2
	= \Omega^2 \left[ \frac{1}{\sin^2(\tau\Omega - \psi)} - \cot^2(\tau\Omega - \psi) \right]
	= \Omega^2
	> 0
\end{equation}
under the assumption of $\rho(\Omega, \psi; \tau) > 0$.
Therefore, all the curves $\Gamma_*(\psi; \tau)$ are contained in a linear cone
\begin{equation*}
	\Set{(a, w) \in \R \times (\C \setminus \{0\})}{\abs{a} < \abs{w}}
\end{equation*}
in $(a, w)$-space.

Calculation~\eqref{eq: rho(Omega, psi; tau)^2 - a(Omega, psi; tau)^2} also shows that each curve $\Gamma_*(\psi; \tau)$ does not have a self-intersection, namely, parametrization~\eqref{eq: parametrization of Gamma_*(psi; tau)} gives a one-to-one correspondence between $(1/\tau)I_*(\psi)$ and $\Gamma_*(\psi; \tau)$.
We note that this is a natural consequence from Theorem~\ref{thm: tau-sequence, angular frequency}.

We next study the curve $\Gamma_\mathrm{c}(\psi; \tau)$ in more detail.
Here the following remark is useful.

\begin{remark}
$\tau\Omega \in I_\mathrm{c}(\psi)$ can be written as $\tau{\abs{\Omega}} \in (0, \abs{\psi})$.
By combining this and
\begin{equation*}
	a(\Omega, \psi; \tau) = a(-\Omega, -\psi; \tau)
	\amd
	\rho(\Omega, \psi; \tau) = \rho(-\Omega, -\psi; \tau),
\end{equation*}
$\Gamma_\mathrm{c}(\psi; \tau) = \Gamma_\mathrm{c}({\abs{\psi}}; \tau)$ holds.
\end{remark}

The above remark shows that we only have to consider the case $\psi > 0$ to study $\Gamma_\mathrm{c}(\psi; \tau)$.

\begin{lemma}
Let $\psi \in (-\pi, \pi]$ and $\tau > 0$ be given.
Then for all $\Omega$ satisfying $\tau\Omega \in I_\mathrm{c}(\psi)$, 
\begin{equation*}
	a(\Omega, \psi; \tau) > \rho(\Omega, \psi; \tau)\cos\psi
\end{equation*}
holds.
\end{lemma}

\begin{proof}
We only have to consider the case $\psi > 0$.
From \eqref{eq: a(theta, psi; tau), rho(theta, psi; tau)}, it is sufficient to consider the case $\tau = 1$.
Since $\rho(\Omega, \psi; 1) > 0$, the inequality is equivalent to $\cos(\Omega - \psi) > \cos\psi$.
Since $\Omega \in (0, \psi)$, i.e., $-\psi < \Omega - \psi < 0$, $\cos(\Omega - \psi) > \cos\psi$ is equivalent to
\begin{equation*}
	-\Omega + \psi < \psi,
\end{equation*}
which trivially holds.
\end{proof}

The above lemma means that for any $\psi \in (-\pi, \pi]$ and any $\tau > 0$, the curve $\Gamma_\mathrm{c}(\psi; \tau)$ is contained in the subset $D_\mathrm{c}$, which is a linear cone and the domain of definition of the critical delay function.

\subsection{``One-to-one correspondence'' and ``ordering"}

In this subsection, we treat a special ``ordering'' for the curves $\Gamma_*(\psi; \tau)$.

\begin{notation}
Let $C$ be a linear cone (in a linear topological space over $\R$) and $\Gamma, \Gamma' \subset C$ be nonempty subsets.
We write $\Gamma \prec \Gamma'$ if the following condition is satisfied:
For every $v \in C$, there exists a unique pair $(s, s')$ of positive numbers such that $s < s'$, $sv \in \Gamma$, and $s'v \in \Gamma'$.
\end{notation}

The above concept should be compared with \cite[page 334]{Kiss--Krauskopf 2010}.

We first consider correspondences between the curves $\Gamma_*(\psi; \tau)$ and $\tau$-values, which will be useful for determining the $\prec$-ordering of the family of curves $(\Gamma_*(\psi; \tau))_{* \in \{\mathrm{c}\} \cup (\Z \setminus \{0\})}$.

\begin{lemma}\label{lem: tau-values and parametrized curves}
Let $\psi \in (-\pi, \pi]$ and $\tau > 0$ be given.
Let
\begin{equation*}
	a \coloneqq a(\Omega, \psi; \tau)
	\amd
	w \coloneqq \rho(\Omega, \psi; \tau)\e^{\iu\psi}
\end{equation*}
for some $\tau\Omega \in I_\mathrm{c}(\psi) \cup \bigcup_{k \in \Z \setminus \{0\}} I_k(\psi)$.
Then the following equivalences hold:
\begin{enumerate}[label=\textup{\arabic*.}]
\item $\tau = \tau_\mathrm{c}(a, w)$ if and only if $\tau\Omega \in I_\mathrm{c}(\psi)$.
\item For each integer $n \ge 1$, $\tau = \tau_n^+(a, w)$ if and only if $\tau\Omega \in I_n(\psi)$.
\item For each integer $n \ge 1$, $\tau = \tau_n^-(a, w)$ if and only if $\tau\Omega \in I_{-n}(\psi)$.
\end{enumerate}
\end{lemma}

\begin{proof}
We prove the statement 1 when $\psi \ne 0$.

(Only-if-part).
By definition, we have
\begin{equation*}
	\tau = \frac{1}{\Omega(a, w)} \left[ \abs{\psi} - \arccos{\left( \frac{a}{\abs{w}} \right)} \right],
\end{equation*}
where $\Omega(a, w) = \abs{\Omega}$ holds from \eqref{eq: rho(Omega, psi; tau)^2 - a(Omega, psi; tau)^2}.
This shows $0 < \tau{\abs{\Omega}} < \abs{\psi}$, which is equivalent to $\tau\Omega \in I_\mathrm{c}(\psi)$.

(If-part).
Since
\begin{equation*}
	\frac{a}{\abs{w}}
	= \cos(\tau\Omega - \psi)
	= \cos(\tau{\abs{\Omega}} - \abs{\psi}),
\end{equation*}
we have $\arccos(a/{\abs{w}}) = -\tau{\abs{\Omega}} + \abs{\psi}$ because $\tau{\abs{\Omega}} \in I_\mathrm{c}(\abs{\psi})$.
By using this and \eqref{eq: rho(Omega, psi; tau)^2 - a(Omega, psi; tau)^2}, we obtain
\begin{equation*}
	\tau_\mathrm{c}(a, w)
	= \frac{1}{\Omega(a, w)} \left[ \abs{\psi} - \arccos{\left( \frac{a}{\abs{w}} \right)} \right]
	= \frac{1}{\abs{\Omega}} \bigl[ \abs{\psi} - (-\tau{\abs{\Omega}} + \abs{\psi}) \bigr]
	= \tau.
\end{equation*}

The above argument shows that $\tau = \tau_\mathrm{c}(a, w)$ is impossible when $\psi = 0$.
Therefore, this completes the proof of the statement 1.
The proofs of the statements 2 and 3 are similar in view of the expressions of $\tau^\pm_n(a, w)$.
Therefore, they can be omitted.
\end{proof}

We note that the above equivalences are natural consequences from Theorem~\ref{thm: tau-sequence, angular frequency}.
The following theorems are related to Lemma~\ref{lem: tau-values and parametrized curves}.

\begin{theorem}\label{thm: positive scalar and tau-value, Gamma_c(psi; tau)}
Let $\psi \in (-\pi, \pi]$ and $\tau > 0$ be given.
Then for every $(a, w) \in D_\mathrm{c}$ and every $s > 0$, $(sa, sw) \in \Gamma_\mathrm{c}(\psi; \tau)$ if and only if
\begin{equation*}
	s = \frac{1}{\tau} \tau_\mathrm{c}(a, w)
	\amd
	\Arg(w) = \psi
\end{equation*}
hold.
\end{theorem}

\begin{proof}
We only have to consider the case $\psi \ne 0$.

(Only-if-part).
$(sa, sw) \in \Gamma_\mathrm{c}(\psi; \tau)$ is equivalent to
\begin{equation*}
	sa = a(\Omega, \psi; \tau)
	\amd
	sw = \rho(\Omega, \psi; \tau)\e^{\iu\psi}
\end{equation*}
for some $\Omega \in (1/\tau)I_\mathrm{c}(\psi)$.
Then by applying Lemma~\ref{lem: tau-values and parametrized curves}, we necessarily have $\tau = \tau_\mathrm{c}(sa, sw)$, where
\begin{equation*}
	\tau_\mathrm{c}(sa, sw)
	= \frac{1}{\Omega(sa, sw)} \left[ \abs{\Arg(w)} - \arccos{\left( \frac{a}{\abs{w}} \right)} \right]
	= \frac{1}{s} \tau_\mathrm{c}(a, w).
\end{equation*}
Therefore, $s = \tau_\mathrm{c}(a, w)/\tau$ is obtained.

(If-part).
Let
\begin{equation*}
	\Omega \coloneqq \Omega(sa, sw) = s\Omega(a, w).
\end{equation*}
Then we have
\begin{equation*}
	\tau
	= \frac{1}{s}\tau_\mathrm{c}(a, w)
	= \frac{1}{s\Omega(a, w)} \left[ \abs{\psi} - \arccos{\left( \frac{a}{\abs{w}} \right)} \right],
\end{equation*}
which implies
\begin{equation*}
	\tau{\abs{\Omega}}
	= \abs{\psi} - \arccos{\left( \frac{a}{\abs{w}} \right)}
	\in (0, \abs{\psi})
	= I_\mathrm{c}(\abs{\psi}).
\end{equation*}
By using
\begin{equation*}
	\sin{\left( \arccos{\left( \frac{a}{\abs{w}} \right)} \right)}
	= \frac{1}{\abs{w}}\Omega(a, w),
\end{equation*}
we obtain
\begin{align*}
	a(\Omega, \psi; \tau)
	&= s\Omega(a, w) \cdot \frac{a/{\abs{w}}}{\Omega(a, w)/{\abs{w}}}
	= sa, \\
	\rho(\Omega, \psi; \tau)
	&= \frac{s\Omega(a, w)}{\Omega(a, w)/{\abs{w}}}
	= s{\abs{w}}.
\end{align*}
Therefore, $(sa, sw) \in \Gamma_\mathrm{c}(\psi; \tau)$ is concluded.

This completes the proof.
\end{proof}

The following is a corollary of Theorem~\ref{thm: positive scalar and tau-value, Gamma_c(psi; tau)}.

\begin{corollary}\label{cor: tau_c(a, w) > tau, Gamma_c}
For each $\psi \in (-\pi, \pi]$ and each $\tau > 0$, we have
\begin{equation*}
	\Set{(a, w) \in D_\mathrm{c}}{\text{$\Arg(w) = \psi$ and $\tau_\mathrm{c}(a, w) > \tau$}}
	= \bigcup_{s > 1} \frac{1}{s} \Gamma_\mathrm{c}(\psi; \tau).
\end{equation*}
Consequently,
\begin{equation*}
	\Set{(a, w) \in D_\mathrm{c}}{\tau_\mathrm{c}(a, w) > \tau}
	= \bigcup_{\psi \in (-\pi, \pi]} \bigcup_{s > 1} \frac{1}{s} \Gamma_\mathrm{c}(\psi; \tau)
\end{equation*}
holds.
\end{corollary}

\begin{proof}
($\subset$).
Let $(a, w) \in D_\mathrm{c}$ be chosen so that $\Arg(w) = \psi$ and $\tau_\mathrm{c}(a, w) > \tau$.
Let
\begin{equation*}
	s \coloneqq \frac{1}{\tau} \tau_\mathrm{c}(a, w).
\end{equation*}
Then $s > 1$ holds by the assumption.
Applying Theorem~\ref{thm: positive scalar and tau-value, Gamma_c(psi; tau)}, we have $(sa, sw) \in \Gamma_\mathrm{c}(\psi; \tau)$.
Therefore, 
\begin{equation*}
	(a, w) \in \bigcup_{s' > 1} \frac{1}{s'} \Gamma_\mathrm{c}(\psi; \tau)
\end{equation*}
holds.

($\supset$).
Let $(a, w) \in (1/s)\Gamma_\mathrm{c}(\psi; \tau)$ for some $s > 1$.
This means $(sa, sw) \in \Gamma_\mathrm{c}(\psi; \tau)$, which implies
\begin{equation*}
	s = \frac{1}{\tau} \tau_\mathrm{c}(a, w)
\end{equation*}
from Theorem~\ref{thm: positive scalar and tau-value, Gamma_c(psi; tau)}.
Therefore, the inequality $\tau_\mathrm{c}(a, w) > \tau$ is concluded.

This completes the proof.
\end{proof}

In the similar way to the proof of Theorem~\ref{thm: positive scalar and tau-value, Gamma_c(psi; tau)} by using the expressions of $\tau^\pm_n(a, w)$, the following theorem is obtained.
The proof can be omitted.

\begin{theorem}\label{thm: positive scalar and tau-value, Gamma_k(psi; tau)}
Let $\psi \in (-\pi, \pi]$, $\tau > 0$, and $k \in \Z \setminus \{0\}$ be given.
Then for every $(a, w) \in \R \times (\C \setminus \{0\})$ satisfying $\abs{a} < \abs{w}$ and every $s > 0$, $(sa, sw) \in \Gamma_k(\psi; \tau)$ if and only if
\begin{equation*}
	s = \frac{1}{\tau} \tau_{\abs{k}}^\pm(a, w)
	\amd
	\Arg(w) = \psi
\end{equation*}
holds.
Here $+$ sign corresponds to the case $k > 0$, and $-$ sign corresponds to the case $k < 0$.
\end{theorem}

We finally obtain the following $\prec$-ordering results.

\begin{corollary}\label{cor: ordering}
Let $\psi \in (-\pi, \pi) \setminus \{0\}$ and $\tau > 0$ be given.
Then for any distinct pair $(*, *')$ in $\{\mathrm{c}\} \cup (\Z \setminus \{0\})$, the curves $\Gamma_*(\psi; \tau)$ and $\Gamma_{*'}(\psi; \tau)$ do not have an intersection.
Furthermore, the following statements hold:
\begin{itemize}
\item If $\psi \in (0, \pi)$, then we have
\begin{gather*}
	\Gamma_{-n}(\psi; \tau) \prec \Gamma_n(\psi; \tau) \prec \Gamma_{-(n + 1)}(\psi; \tau)
	\mspace{20mu}
	(n \ge 1) \\
\intertext{and}
	\Gamma_\mathrm{c}(\psi; \tau) \prec \Gamma_{-1}(\psi; \tau) \cap D_\mathrm{c}.
\end{gather*}
\item If $\psi \in (-\pi, 0)$, then we have
\begin{gather*}
	\Gamma_n(\psi; \tau) \prec \Gamma_{-n}(\psi; \tau) \prec \Gamma_{n + 1}(\psi; \tau)
	\mspace{20mu}
	(n \ge 1) \\
\intertext{and}
	\Gamma_\mathrm{c}(\psi; \tau) \prec \Gamma_1(\psi; \tau) \cap D_\mathrm{c}.
\end{gather*}
\end{itemize}
\end{corollary}

\begin{proof}
For each $* \in \{\mathrm{c}\} \cup (\Z \setminus \{0\})$, let
\begin{equation*}
	\tau_*(a, w) \coloneqq
	\begin{cases}
		\tau_\mathrm{c}(a, w) & (* = \mathrm{c}), \\
		\tau_k^+(a, w) & (* = k > 0), \\
		\tau_{-k}^-(a, w) & (* = k < 0).
	\end{cases}
\end{equation*}
Here $(a, w) \in D_\mathrm{c}$ when $* = \mathrm{c}$, and $(a, w)$ satisfies $\abs{a} < \abs{w}$ when $* = k \ne 0$.
From Lemma~\ref{lem: tau-values and parametrized curves}, the existence of intersection of the distinct curves $\Gamma_*(\psi; \tau)$ and $\Gamma_{*'}(\psi; \tau)$ necessarily implies
\begin{equation*}
	\tau = \tau_*(a, w) = \tau_{*'}(a, w)
\end{equation*}
for some $(a, w)$.
However, this is impossible because of Lemmas~\ref{lem: ordering, real a, 0 < Arg(w) < pi} and \ref{lem: ordering, real a, -pi < Arg(w) < 0}.
The $\prec$-ordering results are also consequences of Theorems~\ref{thm: positive scalar and tau-value, Gamma_c(psi; tau)}, \ref{thm: positive scalar and tau-value, Gamma_k(psi; tau)}, Lemmas~\ref{lem: ordering, real a, 0 < Arg(w) < pi} and \ref{lem: ordering, real a, -pi < Arg(w) < 0} by the definition of $\prec$-ordering.
This completes the proof.
\end{proof}

The result for the cases $\psi = 0$ or $\psi = \pi$ are special.

\begin{corollary}\label{cor: ordering, psi = 0, pi}
Let $\psi \in \{0, \pi\}$ and $\tau > 0$ be given.
Then the following statements hold:
\begin{enumerate}[label=\textup{\arabic*.}]
\item If $\psi = 0$, then for all $n \ge 1$,
\begin{equation*}
	\Gamma_n(\psi; \tau) = \Gamma_{-n}(\psi; \tau) \prec \Gamma_{n + 1}(\psi; \tau) = \Gamma_{-(n + 1)}(\psi; \tau)
\end{equation*}
holds.
\item If $\psi = \pi$, then for all $n \ge 1$,
\begin{equation*}
	\Gamma_n(\psi; \tau) = \Gamma_{-(n + 1)}(\psi; \tau) \prec \Gamma_{n + 1}(\psi; \tau) = \Gamma_{-(n + 2)}(\psi; \tau)
\end{equation*}
holds.
Furthermore, $\Gamma_\mathrm{c}(\psi; \tau) \prec \Gamma_{-1}(\psi; \tau) \cap D_\mathrm{c}$ holds.
\end{enumerate}
\end{corollary}

The proof is similar to the proof of Corollary~\ref{cor: ordering} by using Lemma~\ref{lem: ordering, real a and w}.
Therefore, it can be omitted.

\section{Discussions}\label{sec: discussions}

As is discussed in Introduction, the study of the stability region of Eq.~\eqref{eq: TE} for the imaginary $a$ case is complicated because the critical delay does not exist in general and stability switches may occur (see \cite[Theorems 3.2 and 3.3]{Nishiguchi 2016}).
Nevertheless, the method of consideration by using critical delay should work in this situation by solving appropriately obtained inequalities on $\tau > 0$.
This is a possible future research.

Another direction of a future research is to study characteristic equations of differential equations with multiple delay parameters under the perspective of critical delay.
When the number of delay parameters is two, we are going to consider a transcendental equation of the form
\begin{equation*}
	z + a - w_1\e^{-\tau_1z} - w_2\e^{-\tau_2z} = 0,
\end{equation*}
where $\tau_1, \tau_2 > 0$ are delay parameters.
See \cite{Nussbaum 1978 Memoirs}, \cite{Hale--Huang 1993}, \cite{Belair--Campbell 1994}, \cite{Mahaffy--Joiner--Zak 1995}, \cite{Shayer--Campbell 2000}, \cite{Nussbaum 2002 Handbook of DS}, \cite{Ruan--Wei 2003}, \cite{Mahaffy--Busken 2015}, \cite{Bortz 2016}, and \cite{Breda--Menegonand--Nonino 2018} for studies of the above transcendental equation, for example.
For studies of the stability condition of transcendental equations with multiple delays in the delay parameter space, e.g., see \cite{Hale--Huang 1993}, \cite{Gu--Niculescu--Chen 2005}, and \cite{Jarlebring 2009}.
See also \cite{Sipahi--Niculescu--Abdallah--Michiels--Gu 2011} for a survey article.

It would be also interesting to develop a method to find the stability region without resorting to the explicit expression of the critical delay.
This might be possible because the critical delay is considered to be an implicit function.
This consideration would have a similarity with the method by the Lambert $W$ function in \cite{Nishiguchi 2016} since the Lambert $W$ function is an inverse function.

\section*{Acknowledgment}

This work was supported by the Research Institute for Mathematical Sciences for an International Joint Usage/Research Center located in Kyoto University and JSPS Grant-in-Aid for Young Scientists Grant Number JP19K14565.

\appendix

\section{Analysis based on Lambert $W$ function}\label{sec: Lambert W function}

In this section, we summarize a route to Theorem~\ref{thm: T(a, w), real a} based on the Lambert $W$ function.

\subsection{General results about Lambert $W$ function}

Since the set of all roots of Eq.~\eqref{eq: TE} is expressed by \eqref{eq: set of all roots of z + a - we^{-tau z} = 0}
\begin{equation*}
	\frac{1}{\tau} W(\tau w\e^{\tau a}) - a,
\end{equation*}
all the roots of Eq.~\eqref{eq: TE} have negative real parts if and only if $\Re(z) < \tau\Re(a)$ holds for all $z \in W(\tau w\e^{\tau a})$.
For this type of threshold condition, the following result is obtained in \cite[Lemma 3.1]{Nishiguchi 2016}.

\begin{theorem}[\cite{Nishiguchi 2016}]\label{thm: Lambert W function}
Let $\zeta \in \C$ and $\sigma \in \R$ be given.
Then $\Re(z) < \sigma$ holds for all $z \in W(\zeta)$ if and only if $\zeta$ and $\sigma$ satisfy one of the following conditions:
\begin{enumerate}[label=\textup{(\roman*)}]
\item $\sigma\e^{\sigma} > \abs{\zeta}$.
\item $-{\abs{\zeta}} < \sigma\e^{\sigma} \le \abs{\zeta}$ and
\begin{equation*}
	\abs{\Arg(\zeta)} > \arccos{\left( \frac{\sigma\e^{\sigma}}{\abs{\zeta}} \right)} + \sqrt{({\abs{\zeta}}\e^{-\sigma})^2 - \sigma^2}.
\end{equation*}
\end{enumerate}
\end{theorem}

\begin{remark}
In \cite[Lemma 3.1]{Nishiguchi 2016}, the condition $\zeta \ne 0$ is presumed.
However, since $W(0) = \{0\}$, the property that $\Re(z) < \sigma$ holds for all $z \in W(0)$ is equivalent to $\sigma > 0$.
Therefore, the case $\zeta = 0$ can be included in the statement.
\end{remark}

The following is a corollary of Theorem~\ref{thm: Lambert W function}, which is not stated in \cite{Nishiguchi 2016}.

\begin{corollary}
Let $\zeta \in \C$ and $\sigma \in \R$ be given.
Then $\Re(z) \ge \sigma$ holds for some $z \in W(\zeta)$ if and only if $\zeta$ and $\sigma$ satisfy one of the following conditions:
\begin{enumerate}[label=\textup{(\roman*)}, start=3]
\item $\sigma\e^{\sigma} \le -{\abs{\zeta}}$.
\item $-{\abs{\zeta}} < \sigma\e^{\sigma} \le \abs{\zeta}$ and
\begin{equation*}
	\abs{\Arg(\zeta)} \le \arccos{\left( \frac{\sigma\e^{\sigma}}{\abs{\zeta}} \right)} + \sqrt{({\abs{\zeta}}\e^{-\sigma})^2 - \sigma^2}.
\end{equation*}
\end{enumerate}
\end{corollary}

\begin{proof}
From Theorem~\ref{thm: Lambert W function}, $\Re(z) \ge \sigma$ holds for some $z \in W(\zeta)$ if and only if both of the conditions (i) and (ii) in Theorem~\ref{thm: Lambert W function} does not hold.
Here
\begin{itemize}
\item the condition (i) in Theorem~\ref{thm: Lambert W function} does not hold if and only if $\sigma\e^{\sigma} \le |\zeta|$,
\item the condition (ii) in Theorem~\ref{thm: Lambert W function} does not hold if and only if $\sigma\e^{\sigma} \not\in (-{\abs{\zeta}}, \abs{\zeta}]$ or (iv) holds.
\end{itemize}
Therefore, the equivalence is obtained.
\end{proof}

\subsection{Necessary and sufficient conditions}

By applying Theorem~\ref{thm: Lambert W function} with $\zeta = \tau w\e^{\tau a}$ and $\sigma = \tau\Re(a)$, the following result is immediately obtained in \cite[Theorem 1.2]{Nishiguchi 2016}.

\begin{theorem}[\cite{Nishiguchi 2016}]\label{thm: TE, complex a and w}
Suppose $a, w \in \C$.
Then all the roots of Eq.~\eqref{eq: TE} have negative real parts, i.e., $\tau \in T(a, w)$, if and only if the parameters $a$, $w$, and $\tau$ satisfy one of the following conditions:
\begin{enumerate}[label=\textup{(\roman*)}]
\item $\Re(a) > \abs{w}$.
\item $-{\abs{w}} < \Re(a) \le \abs{w}$ and
\begin{equation}\label{eq: inequality on tau}
	\arccos \bigl( \cos(\tau\Im(a) + \Arg(w)) \bigr) > \arccos{\left( \frac{\Re(a)}{\abs{w}} \right)} + \tau\sqrt{\abs{w}^2 - \Re(a)^2}.
\end{equation}
\end{enumerate}
\end{theorem}

Eq.~\eqref{eq: TE} with complex $a$ and $w$ has been investigated by many authors (e.g., see \cite{Sherman 1952}, \cite{Barwell 1975}, \cite{Braddock--Driessche 1975/76}, \cite{Maset 2000}, \cite{Cahlon--Schmidt 2002}, \cite{Wei--Zhang 2004}, \cite{Shinozaki--Mori 2006}, \cite{Khokhlova--Kipnis--Malygina 2011}, \cite{Breda 2012}).
However, as far as we know, the necessary and sufficient condition given in Theorem~\ref{thm: TE, complex a and w} has not been obtained before \cite{Nishiguchi 2016}.

\begin{remark}
Let $\sigma \in \R$ be given.
By letting $z' \coloneqq z - \sigma$ in the transcendental equation $z\e^z = \zeta$, the equation becomes
\begin{equation}\label{eq: from Lambert W to TE}
	z' + \sigma - \zeta \e^{-\sigma} \e^{-z'} = 0,
\end{equation}
where $\Re(z) < \sigma$ if and only if $\Re(z') < 0$.
Then by applying Theorem~\ref{thm: TE, complex a and w} to Eq.~\eqref{eq: from Lambert W to TE}, the statement of Theorem~\ref{thm: Lambert W function} is obtained.
This means that Theorems~\ref{thm: Lambert W function} and \ref{thm: TE, complex a and w} are logically equivalent.
\end{remark}

\begin{remark}\label{rmk: Re(a) = abs{w}}
When $\Re(a) = \abs{w}$, the right-hand side of inequality~\eqref{eq: inequality on tau} is equal to $0$.
Therefore, the condition (ii) in Theorem~\ref{thm: TE, complex a and w} becomes
\begin{equation*}
	\arccos \bigl( \cos(\tau\Im(a) + \Arg(w)) \bigr) > 0,
\end{equation*}
i.e., $\tau\Im(a) + \Arg(w) \not\in 2\pi\Z$.
\end{remark}

Inequality~\eqref{eq: inequality on tau} in the condition (ii) contains the delay parameter $\tau$.
This makes clear that the case of imaginary $a$ (i.e., $\Im(a) \ne 0$) brings a qualitative change to the condition on $\tau$ for which all the roots of \eqref{eq: TE} have negative real parts.
This fact has been partially known in the literature before \cite{Nishiguchi 2016} (e.g., see \cite{Yeung--Strogatz 1999}, \cite{Wei--Zhang 2004}, \cite{Hovel--Scholl 2005}, \cite{Matsunaga 2009 AMC}, and \cite{Matsunaga 2009b}).
Here the function $\arccos(\cos(\argdot)) \colon \R \to \R$ is the $2\pi$-periodic function satisfying
\begin{equation*}
	\arccos(\cos(\theta)) = \abs{\theta}
\end{equation*}
for all $\theta \in [-\pi, \pi]$.
See \cite[Theorems 3.2 and 3.3]{Nishiguchi 2016} for further details.

The following is a consequence of Theorem~\ref{thm: TE, complex a and w}.

\begin{theorem}\label{thm: TE, real a and complex w}
Suppose $a \in \R$ and $w \in \C$.
Then all the roots of Eq.~\eqref{eq: TE} have negative real parts, i.e., $\tau \in T(a, w)$, if and only if the parameters $a$, $w$, and $\tau$ satisfy one of the following conditions:
\begin{enumerate}[label=\textup{(\roman*)}]
\item $a \ge \abs{w}$ and $a \ne w$.
\item $-{\abs{w}} < a < \abs{w}$ and
\begin{equation}\label{eq: inequality on tau, real a}
	\abs{\Arg(w)} > \arccos{\left( \frac{a}{\abs{w}} \right)} + \tau\sqrt{\abs{w}^2 - a^2}.
\end{equation}
\end{enumerate}
\end{theorem}

\begin{proof}
When $w = 0$, Eq.~\eqref{eq: TE} becomes $z + a = 0$.
Therefore, $-a$ is the only root of Eq.~\eqref{eq: TE}, and the condition $a > 0$ is included in the condition (i) in Theorem~\ref{thm: TE, real a and complex w}.
We now assume $w \ne 0$.
In view of
\begin{equation*}
	\arccos(\cos \Arg(w)) = \abs{\Arg(w)},
\end{equation*}
we only need to check the case $a = \abs{w}$.
In this case, the condition (ii) in Theorem~\ref{thm: TE, complex a and w} becomes $\Arg(w) \ne 0$ from Remark~\ref{rmk: Re(a) = abs{w}}, and $\Arg(w) = 0$ is equivalent to $a = w$.
This completes the proof.
\end{proof}

\begin{remark}
By letting $z' \coloneqq z + \iu\Im(a)$, Eq.~\eqref{eq: TE} becomes
\begin{equation*}
	z' + \Re(a) - w\e^{\iu \tau\Im(a)}\e^{-\tau z'} = 0.
\end{equation*}
Therefore, Theorems~\ref{thm: TE, complex a and w} and \ref{thm: TE, real a and complex w} are logically equivalent.
\end{remark}

A proof of Theorem~\ref{thm: T(a, w), real a} is obtained by using Theorem~\ref{thm: TE, real a and complex w}.

\begin{proof}[Proof of Theorem~\ref{thm: T(a, w), real a} based on Theorem~\ref{thm: TE, real a and complex w}]
We give the proofs of the statements (I), (II), and (III).

\begin{enumerate}
\item[(I)] When $a \ge \abs{w}$ and $a \ne w$, it holds that all the roots of Eq.~\eqref{eq: TE} have negative real parts from Theorem~\ref{thm: TE, real a and complex w} independently from $\tau > 0$.
Therefore, $T(a, w) = (0, \infty)$ holds.
Conversely, we suppose that $T(a, w) = (0, \infty)$ holds.
By using Theorem~\ref{thm: TE, real a and complex w} again, we have ``$a \ge \abs{w}$ and $a \ne w$" or ``$-{\abs{w}} < a < \abs{w}$".
We suppose that the latter condition holds and derive a contradiction.
Theorem~\ref{thm: TE, real a and complex w} shows that inequality~\eqref{eq: inequality on tau, real a} holds for all $\tau > 0$.
However, this is impossible because of
\begin{equation*}
	\lim_{\tau \to \infty} \tau\sqrt{\abs{w}^2 - a^2}
	= \infty.
\end{equation*}
The above argument shows that $T(a, w) = (0, \infty)$ if and only if $a \ge \abs{w}$ and $a \ne w$.
Thus, the statement (I) is obtained in view of Remark~\ref{rmk: absolute stability}.\item[(II)] When $w \ne 0$ and $\Re(w) < a < \abs{w}$, inequality~\eqref{eq: inequality on tau, real a} is satisfied if $0 < \tau < \tau_\mathrm{c}(a, w)$.
Therefore, $(0, \tau_\mathrm{c}(a, w)) \subset T(a, w)$ holds from Theorem~\ref{thm: TE, real a and complex w}.
Conversely, we suppose that $T(a, w)$ is a nonempty proper subset of $(0, \infty)$.
From Theorem~\ref{thm: TE, real a and complex w}, we necessarily have $-{\abs{w}} < a < \abs{w}$, which implies $w \ne 0$.
Theorem~\ref{thm: TE, real a and complex w} also shows that $\tau \in T(a, w)$ implies that inequality~\eqref{eq: inequality on tau, real a} holds.
Therefore, we have
\begin{equation*}
	\abs{\Arg(w)} > \arccos{\left( \frac{a}{\abs{w}} \right)},
\end{equation*}
which implies $\Re(w) < a$ because $\cos|_{[0, \pi]}$ is strictly monotonically decreasing.
Then we have $T(a, w) \subset (0, \tau_\mathrm{c}(a, w))$.
The above argument shows the statement (II).
\item[(III)] When $a \le \Re(w)$, it holds that $T(a, w)$ is empty from the statements (I) and (II) in Theorem~\ref{thm: T(a, w), real a}.
When $a > \Re(w)$, the statement (I) in Theorem~\ref{thm: T(a, w), real a} also shows that $T(a, w) = (0, \infty)$ if $a \ge \abs{w}$.
If $a < \abs{w}$, $w$ is necessarily nonzero, and the statement (II) in Theorem~\ref{thm: T(a, w), real a} shows that $T(a, w)$ is a nonempty proper subset of $(0, \infty)$.
Therefore, the statement (III) is obtained.
\end{enumerate}
This completes the proof.
\end{proof}

Theorem~\ref{thm: TE, real a and complex w} is also obtained from Theorem~\ref{thm: T(a, w), real a}.
Therefore, these theorems are logically equivalent.

\end{document}